\documentclass[a4paper, 11pt]{article}
 
\usepackage{a4wide}
\usepackage{tikz,algorithm, algorithmic, amsmath, amsfonts, amsthm, eucal, pifont,comment}
\usepackage{subfigure}
\usepackage[font={small,it}]{caption}

\usetikzlibrary{trees}

\newcommand{\ttk}{\tilde{\tau}_k}

\newcommand{\T}{\mathcal{T}}

\newlength{\textlarg}

\usepackage[normalem]{ulem}	
\usepackage{cancel}		

\newtheorem{theorem}{Theorem}
\newtheorem{lemma}[theorem]{Lemma}

\newtheorem{corollary}[theorem]{Corollary}
\newtheorem{proposition}[theorem]{Proposition}

\definecolor{gris}{gray}{0.75}

\newcommand{\greybox}[1]{#1}

\title{Some families of trees arising in permutation analysis\footnote{A preliminary version of this work appeared in the extended abstract~\cite{ShortVersion}.}}
\author{Mathilde Bouvel\footnote{LaBRI/CNRS, Universit\'e Bordeaux, and Institut f\"ur Mathematik, Universit\"at Z\"urich.}, 
Marni Mishna\footnote{Dept. Mathematics, Simon Fraser University, Burnaby, Canada.}, 
Cyril Nicaud\footnote{Laboratoire d'Informatique Gaspard Monge (LIGM), Universit\'e Paris-Est, Marne-la-Vall\'ee.}}

\date{}

\begin{document}
\maketitle
\begin{abstract}
We extend classical results on simple varieties of trees 
(asymptotic enumeration, average behavior of tree parameters)
to trees counted by their number of leaves. 
Motivated by genome comparison of related species, 
we then apply these results to strong interval trees with a restriction on the arity of prime nodes. 
Doing so, we describe a filtration of the set of permutations based on their strong interval trees. 
This filtration is also studied from a purely analytical point of view, 
thus illustrating the convergence of analytic series towards a non-analytic limit 
at the level of the asymptotic behavior of their coefficients. 
\end{abstract}


\bigskip
\section{Introduction}
\label{sec:Introduction}
The idea of viewing permutations as enriched trees has been around for
several decades in different research communities. For example, 
the recent enumerative study~\cite{AlAt05} of pattern avoiding 
permutations, in which \emph{(substitution) decomposition trees} 
play a crucial role. Also, the analysis of some sorting algorithms is 
very linked to tree representations of permutations: \emph{PQ 
trees}~\cite{BoLu76} appear in the context of graph algorithms and
\emph{strong interval trees} arise in comparative genomics~\cite[and
references therein, for instance]{BoChMiRo11}.

\medskip

The focus of this article is the study of strong interval trees (a.k.a. decomposition trees). 
Their typical shape (under the uniform distribution) has been described in~\cite{BoChMiRo11}, 
showing that they have a very flat, somehow degenerate, shape. 
Strong interval trees are an essential tool to model and study genome rearrangements. 
But contrary to what this average shape shows, 
the trees associated to permutations that arise in the comparison of mammalian genomes seem to have a richer, deeper structure. 
This suggest that trees coming from permutations under the uniform distribution do not adequately represent the trees that arise in genomic comparisons. 

In~\cite{BoChMiRo11}, Bouvel~\emph{et al.} considered a subclass of
strong interval trees --selected because they represent what is known
as \emph{commuting scenarios}~\cite{BeBeChPa07}-- that correspond to the
class of \emph{separable permutations}. This is a first step towards a
more relevant model of permutations which arise in genome
comparison. By studying asymptotic enumeration and parameter formulas
for separable permutations, they proved that the complexity of 
the algorithm of~\cite{BeBeChPa07} solving the \emph{perfect sorting by reversals} problem
is polynomial time on separable permutations, whereas this problem is
NP-complete in general. Furthermore they were also able to
describe some average-case properties of the perfect sorting scenarios
for separable permutations.

Ultimately, a clear understanding of the properties possessed by
  the strong interval trees that represent the comparison of actual
  genomes might tell us something about the evolutionary process.
Bouvel~\emph{et al.}~\cite{BoChMiRo11} conclude their study on
separable permutations with a suggestion for the next step: strong
interval trees with degree restrictions on certain internal
nodes. It is a very controlled way to introduce bias in the
distribution of strong interval trees.  This is precisely what we do in
this work; namely, we study strong interval trees where the prime
nodes have a bounded number of children. 
In this work, we focus on the combinatorial analysis
of these restricted sets of trees.
They can be completely understood combinatorially and analytically, 
and so we have access to enumeration and analysis of some tree parameters 
that are ultimately related to the complexity of computing perfect sorting scenarios, or to properties of these scenarios. 

\medskip

Although our initial motivation is the application of combinatorial analysis to 
a better understanding of models for genome rearrangements, 
we believe our study has ramifications of independent interest in analytic combinatorics. 
Indeed, our work reveals a very lush substructure of permutations 
whose study from an analytical point of view allows us to formulate 
new questions on the convergence of sequences of combinatorial series. 

Specifically, we define nested families of trees (which are almost simple varieties of trees) whose limit is the set of all strong interval trees. 
The components are then families of trees to which we are able to apply a very complete set of tools: 
asymptotic analysis, parameter analysis, random generation. 
But the complete set of strong interval trees is not even close to be a simple variety of trees, 
so that these tools are inaccessible to the full class 
without working through the size preserving bijection between strong interval trees and permutations.  
The question that we ask is then to understand at the analytical level the convergence of 
algebraic subclasses of a non-analytic class towards the full class. 
As we explain in details in our work, this question is naturally asked for strong interval trees, 
but it could also be considered for other classes such as $k$-regular graphs~\cite{Gess90} or $\lambda$-terms of bounded unary height~\cite{BGG11}. 
 
\medskip

The organization of this article is as follows. 
First, in Section~\ref{sec:General} we present some very general theorems 
for asymptotic enumeration and parameter analysis in families of trees counted by leaves, 
that are widely applicable.
Then in Section~\ref{sec:CIT} we describe strong interval
trees, a decomposable combinatorial class of trees counted by leaves, in bijection with permutations. 
Next, in Section~\ref{sec:restrictedCIT} we introduce a filtration of strong interval trees, 
where we bound the arity of so-called \emph{prime} nodes. 
As discussed in Section~\ref{sec:restrictedCIT}, this filtration has applications in bio-informatics 
for the study of genome rearrangements. 
But it also witnesses an intriguing analytic phenomenon: 
the convergence of a sequence of (well-behaved and algebraic) families of trees 
towards the (transcendental and non-analytic) class of permutations. 
Section~\ref{sec:filtration} establishes some first results in the exploration of this phenomenon. 

\section{When the size of a tree is the number of leaves}
\label{sec:General}
There are many works which study the average case behavior of 
tree parameters, where the size of a tree is the number 
of internal nodes or of both internal nodes and leaves. 
The generating functions of these trees satisfy a functional
equation of the form $T(z)=z\cdot\Phi(T(z))$, and when $\Phi$ satisfies
certain conditions, such as analyticity, then there are formulas for
inversion, resulting in explicit enumerative results. A class of trees
amenable to this treatment is said to be a \emph{simple variety of trees}. 
The subject is exhaustively treated in~Section VII.3 of~\cite{FlSe09}. 

If, instead, we define the size of a tree as the number of leaves,
the generating function satisfies a relation of the form
$T(z)=z+\Lambda(T(z))$. The same general theorems on inversion still
work, and it suffices to apply them and unravel the results. 
Even though they are less frequent, these
have also been 
studied in the literature, and the applicability
of the inversion lemmas is noted in~Example VII.13 of
\cite{FlSe09}. In this section we do this explicitly. 
In this work, when refering to simple varieties of trees, 
we mean a family of trees where the size is defined as the number of leaves, 
and whose study falls into the scope of the results of the present section. 

Table~\ref{tab:summary} summarizes the results of this section. We
determine asymptotic formulas for number of trees, and several key
parameters. The shape of the formulas are, unsurprisingly, not unlike
those that arise in the study of trees counted by internal nodes.
\begin{table}[t]
\greybox{
\center
\begin{tabular}{lll}
  Asymptotic number of trees with $n$ leaves & $\sqrt{\frac{\rho}{2\pi\Lambda''(\tau)}}\cdot \frac{\rho^{-n}}{n^{3/2}}$\\[4mm]
  The average number of nodes of arity $\kappa$ in trees with $n$ leaves &
  $\frac{\lambda_\kappa\tau^\kappa}{\rho} \cdot n$\\[4mm]
  The average number of internal nodes in trees with $n$ leaves &
  $\frac{\Lambda(\tau)}{\rho}\cdot n=\frac{\tau-\rho}{\rho}\cdot
  n$\\[4mm]
  The average subtree size sum in trees with $n$ leaves &
  $\sqrt{\frac{\pi}{2\rho\Lambda''(\tau)}}\cdot n^{3/2}$\\[2mm]

\end{tabular}}
\caption{A summary of parameters of trees given by $T = z +
  \Lambda(T)$. The value $\tau$ is the unique solution to
  $\Lambda'(\tau)=1$ between $0$ and $R_{\Lambda}<1$, and $\rho=\tau-\Lambda(\tau)$. 
} 
\label{tab:summary}
\end{table}

\subsection{Asymptotic number of trees}
Our entire analysis is roughly a consequence of the Analytic Inversion
Lemma and Transfer Theorems.  The version to which we appeal is given
and proved in~\cite{FlSe09}. Citations to original sources may be
found therein.  The following theorem is a slight adaptation of
Proposition~IV.5 and Theorem~VI.6 of~\cite{FlSe09} to combinatorial equations of the
form $\T=\mathcal{Z}+\Lambda(\T)$ instead of
$\T=\mathcal{Z}\cdot\Phi(\T)$.

\begin{theorem}\label{thm:main}
  Let~$\Lambda$ be a function analytic at~$0$, with non-negative
  Taylor coefficients, and such that, near~$0$,
\[
\Lambda(z) = \sum_{n\geq 2} \lambda_n z^n.
\]
Let~$R_\Lambda$ be the radius of convergence of this series. Under the
condition \mbox{$\lim_{x\rightarrow R_\Lambda^{-}} \Lambda'(x) > 1$},
there exists a unique solution $\tau \in (0,R_\Lambda)$ of the
equation~$\Lambda'(\tau) = 1$.

Then, the formal solution $T(z)$ of the equation 
\begin{equation}
T(z) = z+\Lambda(T(z)) \label{eq:main}
\end{equation}
is analytic at $0$, its unique dominant singularity
is at $\rho = \tau - \Lambda(\tau)$ and its expansion near $\rho$ is
\begin{equation*}
T(z) = \tau - \sqrt{\frac{2\rho}{\Lambda''(\tau)}} \sqrt{1-\frac{z}{\rho}} + \mathcal{O}\left(1-\frac{z}{\rho}\right).
\end{equation*}
Moreover, if~$T$ is aperiodic, then one has
\begin{equation*}
[z^n]T(z) \sim \sqrt{\frac{\rho}{2\pi\Lambda''(\tau)}}\cdot \frac{\rho^{-n}}{n^{3/2}}.
\end{equation*}
\end{theorem}

\begin{proof}
The conditions on $\Lambda$ imply that both $\Lambda(x)$ and $\Lambda'(x)$ are increasing functions for $x$ in the real interval $(0,R_\Lambda)$. 
Since $\Lambda'(0) = 0$ and since \mbox{$\lim_{x\rightarrow R_\Lambda^{-}} \Lambda'(x) > 1$}, there exists $R'\in (0,R_\Lambda)$ such that
$\Lambda'(R')>1$. Hence there exists a unique $\tau\in(0,R')$, and thus on $(0,R_{\Lambda})$, such that  $\Lambda'(\tau)=1$.

Now observe that Equation~\ref{eq:main} admits a unique formal power series solution $T(z)$, which has non-negative coefficients,
by bootstrapping the coefficients. By Analytic Inversion~\cite[Lemma IV.2]{FlSe09}, this solution is analytic at $z=0$ and with $T(0)=0$:
Equation~\ref{eq:main} can be rewritten $\Psi(T(z)) = z$, with $\Psi(x) = x-\Lambda(x)$, and $\Psi'(0) \neq 0$.

By Pringsheim's Theorem, a dominant singularity of $T(z)$, if any, lies on the positive real axis. Let $r\in(0,+\infty]$ be the
radius of convergence of $T(z)$ at $0$, and set $T(r)\in (0,+\infty]$ be defined by $T(r) = \lim_{x\rightarrow r^-}T(x)$. Following almost exactly
the proof of Proposition~IV.5 in~\cite[p. 278]{FlSe09}, we get that $T(r)=\tau$. Since  $T$ and $\Psi$ are inverse functions,  we get
by continuity that $T(z)$ has a unique dominant singularity at $\rho = \tau-\Lambda(\tau)$.

The remainder of the proof follows almost readily the one of Theorem~VI.6 in~\cite[p. 405]{FlSe09}, using our specific equations.
\end{proof}

\subsection{Parameter Analysis}

In the case of trees counted by internal nodes, the study of
recursively defined parameters is very straightforward, starting from
generating function equations. We can describe analogous versions for
trees counted by leaves. In particular, we consider additive
parameters, and describe a Modified Iteration Lemma, adapted to our
notion of size. We illustrate the lemma on number of internal nodes,
subtree size sum and number of nodes of a given arity.

\subsubsection{General additive parameters}

Our focus is on tree parameters that can be computed additively by
parameters of subtrees. More precisely, we consider a parameter $\xi(t)$ for
trees $t\in\T$ which satisfy the relation
\[
\xi(t) = \eta(t) + \sum_{j=1}^{\operatorname{deg}(t)}\sigma(t_j),
\]
where $\operatorname{deg}(t)$ is the arity of the root, $t_j$ are its children, 
$\eta$ is a simpler tree parameter, 
and $\sigma$ is either $\xi$ or a simpler tree parameter. 
Let $\Xi(z)$, $H(z)$ and 
$\Sigma(z)$ be the associated cumulative functions of $\xi$, $\eta$
and $\sigma$. That is, 
\[
\Xi(z)=\sum\limits_{t\in\T}\xi(t) z^{|t|}, \quad 
H(z)=\sum\limits_{t\in\T}\eta(t) z^{|t|} \textrm{\quad  and \quad  }
\Sigma(z)=\sum\limits_{t\in\T}\sigma(t) z^{|t|}.
\]

Lemma VII.1 in~\cite{FlSe09} has an analogue for trees counted by
their leaves, and it is proved in a very similar way.
\begin{lemma}[Iteration Lemma for trees counted by their leaves]
\label{lm:inv}
Let~$\T$  be a class of trees satisfying $\T=\mathcal{Z}+\Lambda(\T)$.
The cumulative generating functions are related by
\[
\Xi(z) = H(z) + \Lambda'(T(z))\,\Sigma(z).
\]
In particular, if $\sigma\equiv\xi$, one has \ $
\Xi(z) = \frac{H(z)}{1-\Lambda'(T(z))} = H(z)\cdot T'(z).
$
\end{lemma}

\begin{proof}
Unraveling the definition of $ \xi(t)$, we have 
\[
\Xi(z) = H(z) + \widetilde{\Xi}(z) \textrm{ \ with \ } \widetilde{\Xi}(z) = \sum_{t \in \T} z^{|t|} \sum_{j=1}^{\operatorname{deg}(t)} \sigma(t_j).
\]
Splitting the sum defining $\widetilde{\Xi}(z)$ according to the value $r$ of the degree of the root of $t$, we get:  
\begin{align*}
\widetilde{\Xi}(z) &= \sum_{r \geq 1} \lambda_r z^{|t_1| + \ldots + |t_r|}(\sigma(t_1) + \ldots + \sigma(t_r)) \\ &= 
\sum_{r \geq 1} \lambda_r \left( \sigma(t_1)z^{|t_1|}z^{|t_2| + \ldots + |t_r|} + \ldots + \sigma(t_r) z^{|t_r|}z^{|t_1| + \ldots + |t_{r-1}|} \right)\\ &= 
\sum_{r \geq 1} \lambda_r \times r \times \Sigma(z) T(z)^{r-1}=\Lambda'(T(z)) \,  \Sigma(z).
\end{align*}

In the case  $\sigma\equiv\xi$, $\Xi(z) = \frac{H(z)}{1-\Lambda'(T(z))}$ is derived immediately. 
The last equality is a consequence of $T'(z)(1-\Lambda'(T(z))) = 1$,
which is obtained by differentiating $T(z) = z + \Lambda(T(z))$ with
respect to~$z$.
\end{proof}

Note that if $\sigma\equiv\xi$, the parameter is said to be \emph{recursive}. 
Most basic parameters are recursive, and in what
follows we shall use this case only. 

Note also that when analytic treatment applies,~$T(z)$ has a square-root singularity (see Theorem~\ref{thm:main}),
so that $T'(z)$ has an inverse square-root singularity (by analytic
derivation). Therefore, whenever $H(z)$ tends to a positive real when
$z\rightarrow\rho$ (under some analytic conditions), then the Transfer Theorem 
yields an asymptotic equivalent of the mean value of the parameter of
the form $c\cdot n$. This is for instance the case for the number of
nodes of fixed arity and the number of internal nodes, as shown below.

\subsubsection{Three applications}

\paragraph{Number of nodes with exactly $\kappa$ children.}
We ``mark'' nodes of arity $\kappa$ by setting
\[
\eta(t) = \begin{cases}
1 & \text{if the root of $t$ is of arity }\kappa,\\
0 & \text{otherwise}.
\end{cases}
\]
Hence if $\kappa\geq 2$, $
H(z) = \sum\limits_{t\in \T} \eta(t)z^{|t|} =
\sum\limits_{t_1,\ldots,t_\kappa\in\T}\lambda_\kappa
z^{|t_1|+|t_2|+\ldots+|t_\kappa|} $ 
so that $H(z) = \lambda_\kappa\,T(z)^\kappa$. Now,\footnote{This is the
  only other possibility since there can be no unary nodes in a proper
  specification.} if $\kappa=0$, then $H(z) = z$ which is not interesting
since it is simply counting the number of leaves \emph{i.e.} the size of the
tree.

\noindent By Lemma~\ref{lm:inv}, for any $\kappa\geq 2$ one has~$\Xi(z) =
\lambda_\kappa T(z)^{\kappa}\cdot T'(z)$. Since  the singular expansion of $T(z)$ near $\rho$ is
\begin{equation}\label{eq:gamma}
T(z) = \tau - \gamma \sqrt{1-z/\rho} +o\left(\sqrt{1-z/\rho}\right), \text{with }\gamma=\sqrt{\frac{2\rho}{\Lambda''(\tau)}}
\end{equation}
then near $\rho$, one has $T(z)^\kappa = \tau^\kappa + \mathcal{O}\left(\sqrt{1-z/\rho}\right).$
Using the Singular Differentiation Theorem we have
\[
T'(z) = \frac{\gamma}{2\rho \sqrt{1-z/\rho}} +o\left(\frac1{\sqrt{1-z/\rho}}\right),
\text{ so that }
\Xi(z) = \frac{\lambda_\kappa \gamma\tau^\kappa}{2\rho \sqrt{1-z/\rho}} + o\left(\frac1{\sqrt{1-z/\rho}}\right),
\]
from which we get the asymptotics of the cumulative generating function
\[
[z^n]\Xi(z)\sim\frac{\lambda_\kappa \gamma\tau^\kappa\rho^{-n-1}}{2\sqrt{\pi n}}.
\]
The asymptotics of the average value across all trees of size $n$ is then 
\[
\frac{[z^n]\Xi(z)}{[z^n]T(z)} \sim 
\frac{\lambda_\kappa \gamma\tau^\kappa\rho^{-n-1}}{2\sqrt{\pi n}} \cdot \sqrt{\frac{2\pi\Lambda''(\tau)}{\rho}} \frac{n^{3/2}}{\rho^{-n}}
\sim \frac{\lambda_\kappa\tau^\kappa}{\rho} \cdot n ,
\]
as reported in Table~\ref{tab:summary}.

\paragraph{Number of internal nodes.}
For this parameter, just take the following definition for $\eta$:
\[
\eta(t) = \begin{cases}
0 & \text{if  }t\text{ is just one leaf,} \\
1 & \text{otherwise}.
\end{cases}
\]
One has $H(z) = \sum\limits_{t\in\T} \eta(t)z^{|t|} = T(z)-z$, 
and therefore (with the $\gamma$ of Equation~\eqref{eq:gamma})
\[
\Xi(z) = \left(T(z)-z\right)\,T'(z) = \frac{\gamma(\tau-\rho)}{2\rho \sqrt{1-z/\rho}}   +o\left(\frac1{\sqrt{1-z/\rho}}\right).
\]
It follows that 
\[
[z^n]\Xi(z)\sim\frac{\gamma(\tau-\rho)\rho^{-n-1}}{2\sqrt{\pi n}} \textrm{\quad and \quad }
\frac{[z^n]\Xi(z)}{[z^n]T(z)} \sim \frac{\tau-\rho}{\rho} \cdot n.
\]

\paragraph{Subtree size sum.}
We are interested in the subtree size sum parameter, defined by
$\eta(t) = |t|$. This implies that $H(z)=zT'(z)$, so that
\[
\Xi(z) = z T'(z)^2 = \frac{\gamma^2}{4\rho(1-z/\rho)} +o\left(\frac{1}{1-z/\rho}\right)
\text{\quad and \quad}
[z^n]\Xi(z)\sim\frac{\gamma^2}{4\rho} \cdot \rho^{-n}.
\]
Unlike the two previous examples, this is not an inverse of square-root singularity.
In this case, for the average value of the subtree size sum, we find  
\[
\frac{[z^n]\Xi(z)}{[z^n]T(z)} \sim \frac{\gamma^2}{4\rho} \rho^{-n} \cdot \sqrt{\frac{2\pi\Lambda''(\tau)}{\rho}} \frac{n^{3/2}}{\rho^{-n}}
 \sim \sqrt{\frac{\pi}{2\rho\Lambda''(\tau)}}\cdot n^{3/2},
\]
that is, an asymptotic equivalent in $n^{\frac32}$. 
This behavior is typical for such path length related parameters.

\medskip

There are many other tree parameters that we could consider in a
similar fashion.
\section{Strong Interval Trees}
\label{sec:CIT}
Our interest in trees counted by leaves is spawned by \emph{strong
 interval trees}. Strong interval trees are in a size preserving
bijection with permutations.  They have been introduced in the early
2000's in a bio-informatics context~\cite{HeSt01,BeChdeMRa05}: they
are a very effective data structure for algorithms in reconstruction
of genome evolution scenarios, as we briefly mentioned in
Section~\ref{sec:Introduction}.  Under a different name, and roughly
at the same time, these objects also made their appearance in
combinatorics, in the study of permutation patterns: strong interval
trees (rather called (substitution) decomposition trees) are a tree
representation of the block decomposition of permutations described by
Albert and Atkinson~\cite{AlAt05}.  Although the proper definition of
strong interval trees is relatively recent, it can be traced to older
notions of decomposition (of graphs in particular): it is a close
relative of the modular decomposition trees of permutation
graphs~\cite{BeChdeMRa05} and even has origins in the PQ-trees of
Booth and Lueker~\cite{BoLu76}.

In this section, we review the definition of strong interval trees and the bijection with permutations. 
Then, we turn to a presentation of these objects as a constructible combinatorial class, 
in the flavor of what is done in Section~\ref{sec:General}. 

\subsection{Definition and bijection with permutations}

Strong interval trees are most often defined via the bijection that relates them to permutations. 
Different presentations of this bijection can be found for instance in~\cite{AlAt05,BeChdeMRa05,BoChMiRo11}. 
For the reader who is not familiar with these objects, we review the definition of strong interval trees, 
and the correspondence with permutations below. 

In the context of our work, a permutation of size $n$ is a word containing exactly once each symbol in $\{1,2,\ldots, n\}$. 

An \emph{interval} of a permutation $\sigma$ is a factor of $\sigma$, such that the underlying set of symbols is an interval of integers. 
For instance, $7\,\,9\,\,10\,\,11\,\,13\,\,8\,\,12$ and $3\,\,1\,\,5\,\,4\,\,2$ are intervals of the permutation 
$6\,\,7\,\,9\,\,10\,\,11\,\,13\,\,8\,\,12\,\,3\,\,1\,\,5\,\,4\,\,2$, but $10\,\,11\,\,13$ is not ($12$ is missing). 
For every permutation $\sigma$ of size $n$, the singletons $i$ (for $1\leq i \leq n$) and $\sigma$ itself are intervals of $\sigma$. 
They are called \emph{trivial} intervals of $\sigma$. 

A permutation is said to be~\emph{simple\/} when its only intervals are the trivial ones. 
Note that our convention will be that $1$, $1\,\,2$ and $2\,\,1$ are not simple permutations, although they satisfy the above definition. 
It is immediate to check that there is no simple permutation of size $3$ (each permutation of size $3$ containing an interval of size $2$), 
and that there are $2$ simple permutations of size $4$, namely $2\,\,4\,\,1\,\,3$ and $3\,\,1\,\,4\,\,2$. 
A larger simple permutation is for instance $3\,\,5\,\,7\,\,1\,\,4\,\,2\,\,6$. 
We will go back to the enumeration of simple permutations in the next subsection. 

Two intervals of $\sigma$ \emph{overlap} when their intersection is neither empty nor equal to one of them. 
Returning to our example of $6\,\,7\,\,9\,\,10\,\,11\,\,13\,\,8\,\,12\,\,3\,\,1\,\,5\,\,4\,\,2$, 
the intervals $6\,\,7$ and $7\,\,9\,\,10\,\,11\,\,13\,\,8\,\,12$ overlap (their intersection is $7$), but $10\,\,11$ and $5\,\,4$ don't. 
A \emph{strong} interval of $\sigma$ is an interval that does not overlap any other interval of $\sigma$. 
The trivial intervals are obviously strong. On our running example, the non-trivial strong intervals are 
\[
5\,\,4 \ ; \quad 
3\,\,1\,\,5\,\,4\,\,2 \ ;\quad 
9\,\,10\,\,11 \ ;\quad 
9\,\,10\,\,11\,\,13\,\,8\,\,12 \quad \textrm{ and } \quad 
6\,\,7\,\,9\,\,10\,\,11\,\,13\,\,8\,\,12. 
\]

From their definition, it follows immediately that 
the inclusion order on the set of strong intervals of a permutation $\sigma$ induces a tree structure, 
where the leaves are the singletons, and the root is the $\sigma$ itself. 
This is the \emph{strong interval tree} of $\sigma$. 

From there, and depending on the context, the definition of the strong interval tree may vary. 
For us, these trees are \emph{embedded in the plane}, by imposing the order of the leaves. 
Namely, from left to right, the leaves (corresponding to singletons of $\sigma$) are required to appear in the same order as in $\sigma$. 
The strong interval tree of our running example would then be: 
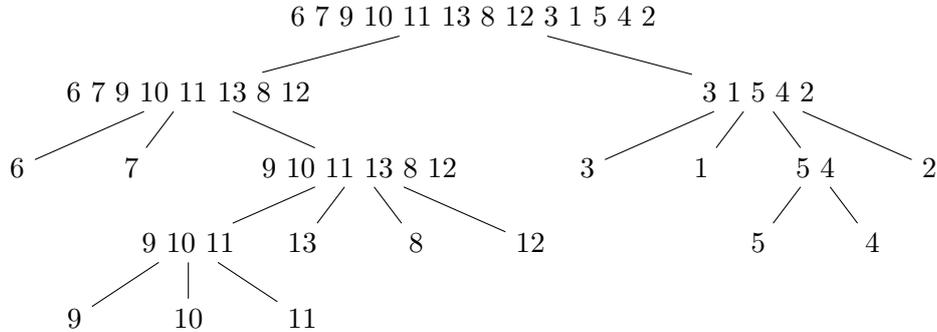
\begin{center}
\begin{tikzpicture}[level distance=1cm]
\node {$6\,\,7\,\,9\,\,10\,\,11\,\,13\,\,8\,\,12\,\,3\,\,1\,\,5\,\,4\,\,2$}
 child{ node{$6\,\,7\,\,9\,\,10\,\,11\,\,13\,\,8\,\,12$} 
   child { node{$6$}
   }
   child { node{$7$}
   }
   child[missing]
   child { node{$9\,\,10\,\,11\,\,13\,\,8\,\,12$}
     child{ node{$9\,\,10\,\,11$}
       child{ node{$9$}
       }
       child{ node{$10$}
       }
       child{ node{$11$}
       }
     }
     child{ node{$13$}
     }
     child{ node{$8$}
     }
     child{ node{$12$}
     }
   }
 }
 child[missing]
 child[missing]
 child[missing]
 child[missing]
 child{ node{$3\,\,1\,\,5\,\,4\,\,2$}
   child{ node{$3$}
   }
   child{ node{$1$}
   }
   child{ node{$5\,\,4$} 
     child {node{$5$}
     }
     child {node{$4$}
     }
   }
   child{ node{$2$}
   }
 };
\end{tikzpicture}
\end{center}

From this tree, there is a last step that we perform before obtaining what we refer to as the strong interval tree in our work. 
It relies on an important remark, proved for instance in~\cite{BXHaPa05}. 
To state it, we first need to describe how to associate a permutation of size $k$ 
to each node of the strong interval tree with $k$ children.

Note first that given several disjoint strong intervals, the natural order on integers induces an order among them: 
the smaller the elements contained in the interval, the smaller the interval itself. 
On our running example, we have for instance that 
$5\,\,4$ is smaller than $8$ which is itself smaller than $9\,\,10\,\,11$. 
Consider now a non-singleton strong interval, corresponding to an internal node of the strong interval tree. 
Its children (assume there are $k$ of them) are also strong intervals, and they are disjoint. 
So they can be ordered as described above. 
To this node of the tree, we associate a permutation $\tau$ of size $k$ built as follows: 
$\tau_i = j$ if the $i$-th child from the left is the $j$-th smallest one. 
For instance, the permutation $\tau$ associated with the node labeled $9\,\,10\,\,11\,\,13\,\,8\,\,12$ in our running example 
is $2\,\, 4 \,\, 1 \,\, 3$ since $8$ is smaller than $9\,\,10\,\,11$, itself smaller than $12$ and $13$. 

The permutations labeling the node enjoy a remarkable property (see~\cite{BXHaPa05}, or in a somewhat different presentation~\cite{AlAt05}): 
they are either increasing ($1 \, \, 2 \ldots k$) or decreasing ($k \ldots 2 \, \, 1$) or simple. 
In the reminder of this article, when speaking about strong interval trees, 
we mean the plane tree whose structure has been described above, 
but whose internal nodes are only labeled by $\oplus$, $\ominus$ 
(corresponding to increasing or decreasing permutations respectively), or by a simple permutation. 
In particular, the leaves carry no label. 
Nodes labeled by $\oplus$ or $\ominus$ are called \emph{linear}, 
whereas those labeled by simple permutations are called \emph{prime}. 
Figure~\ref{fig:example-CIT}$(a)$ shows the strong interval tree of our running example. 
Figure~\ref{fig:example-CIT}$(b)$ represents this tree for a simple permutation. 
What can be observed on this example is true in general: 
the trees corresponding to simple permutations consist of a single prime node labeled by the permutation itself, with pending leaves.

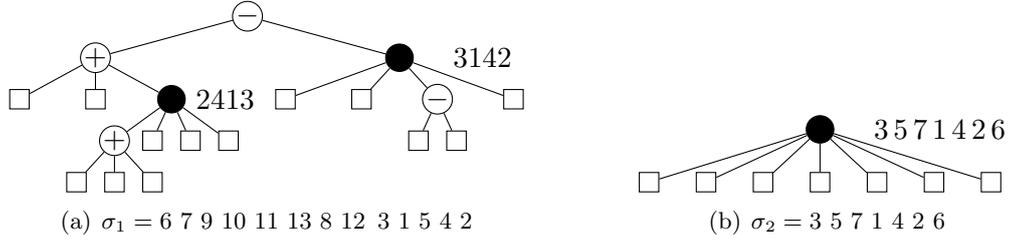
\begin{figure}\center
\mbox{
\subfigure[$\sigma_1=6\,\,7\,\,9\,\,10\,\,11\,\,13\,\,8\,\,12\,\,\,3\,\,1\,\,5\,\,4\,\,2$]{
\begin{tikzpicture}[%
level 1/.style={sibling distance=8cm,level distance=1.1cm},
    level 2/.style={sibling distance=2cm},
    level 3/.style={sibling distance=1cm},
    level 4/.style={sibling distance=1cm},
    scale=0.5
]
\node (root)[circle, draw, inner sep=0pt] {$-$}
 child{node[circle,draw, inner sep=0pt] (L1) {$+$}{
     child{node[draw] (1) {}{}}
     child{node[draw] (2) {}{}}
     child{node[circle, draw, fill, label=right:{$2413$}] (P1) {}{ 
          child{node[circle,draw, inner sep=0pt] (L2) {$+$}{
            child{node[draw] (3) {}{}}
            child{node[draw] (4) {}{}}
            child{node[draw] (5) {}{}}
          }}
          child{node[draw] (6) {}{}}
          child{node[draw] (7) {}{}}
          child{node[draw] (8) {}{}}
          }}
     }}
 child{node[circle, draw, fill, label=right:{\quad $3142$}] (P2) {}{
     child{node[draw] (9) {}{}}
     child{node[draw] (10) {}{}}
     child{node[circle, draw, inner sep=0pt] (L3) {$-$}{
            child{node[draw] (11) {}{}}
            child{node[draw] (12) {}{}}
           }}
     child{node[draw] (13) {}{}}
     }
};
\end{tikzpicture} 
}
\hspace{1cm}
\subfigure[$\sigma_2= 3\,\,5\,\,7\,\,1\,\,4\,\,2\,\,6$]{
\begin{tikzpicture}[%
level 1/.style={sibling distance=1.5cm,level distance=1.4cm},
    level 2/.style={sibling distance=2cm},
    level 3/.style={sibling distance=1cm},
    level 4/.style={sibling distance=1cm},
    scale=0.5
]
\node (root)[circle, draw, fill, label=right:{\quad $3\,5\,7\,1\,4\,2\,6$}] {}
     child{node[draw]  {}{}}
     child{node[draw]  {}{}}
     child{node[draw]  {}{}}
     child{node[draw]  {}{}}
     child{node[draw]  {}{}}
     child{node[draw]  {}{}}
     child{node[draw]  {}{}}
;
\end{tikzpicture} 
}}
\caption{Two permutations and their associated strong interval trees}
\label{fig:example-CIT}
\end{figure}

\medskip

In a strong interval tree, it is impossible for a node labeled by $\oplus$ (resp. $\ominus$) to have a child carrying the same label. 
This property appears (although in disguise) in~\cite{AlAt05}, but is simply proved by contradiction: 
assuming that a parent and a child both carry the label $\oplus$ (resp. $\ominus$) contradicts that the child is a strong interval 
(indeed, it overlaps an interval resulting from the union of one of its own children with one of its siblings). 

With this in mind, strong interval trees are now just plane trees, 
where internal nodes are of arity at least $2$ and carry labels $\oplus$, $\ominus$ or $\tau$ for any simple permutation $\tau$, 
with the additional conditions that a node labeled by $\oplus$ (resp. $\ominus$) does not have a child carrying the same label, 
and that the number of children of a node labeled by a simple permutation $\tau$ is exactly the size of $\tau$. 

It turns out (and the proof follows immediately from~\cite{AlAt05}) that any such tree is the strong interval tree of a permutation. 
Moreover, the above construction provides a bijection between permutations and strong interval trees. 
Note that the bijection is completely constructive, 
and that it can be computed in linear time, although this is quite difficult to achieve, see~\cite{BeChdeMRa05}. 


\subsection{Strong interval trees as a constructible class}

From now on, we denote $\mathcal{P}$ the class of strong interval trees. 
As we have seen above, this class is a set of trees 
where some internal nodes are signed and others are enriched with a simple permutation. 
More precisely, the characterization of strong interval trees given above can be rephrased as show in the following theorem. 

\begin{theorem}[Reformulated from~\cite{AlAt05}]
  The class of permutations is in a size-preserving bijection with
  the combinatorial class~$\mathcal{P}$ of strong interval trees. 
  These are enriched trees defined
  by the following relations, where size is given by the number of
  leaves:
\begin{equation}\label{eq:initialsystem}
\begin{aligned}
\mathcal{P} &= \mathcal{Z}_{\Box} +\, \mathcal{N}_{\oplus}\cdot
\operatorname{Seq}_{\geq 2}{\mathcal{U_{\oplus}}} +\, \mathcal{N}_{\ominus}\cdot
\operatorname{Seq}_{\geq 2}{\mathcal{U_{\ominus}}}+\,
\mathcal{N}_{\bullet}\cdot\, S(\mathcal{P}),\\
\mathcal{U_{\oplus}} &= \mathcal{Z}_{\Box} +\,
\mathcal{N}_{\ominus}\cdot\operatorname{Seq}_{\geq 2}{\mathcal{U}_{\ominus}} +\,
\mathcal{N}_{\bullet}\cdot\, S(\mathcal{P}),\\
\mathcal{U_{\ominus}} &= \mathcal{Z}_{\Box} +\,
\mathcal{N}_{\oplus}\cdot\operatorname{Seq}_{\geq 2}{\mathcal{U}_{\oplus}} +\,
\mathcal{N}_{\bullet}\cdot\, S(\mathcal{P}).
\end{aligned}
\end{equation}
Above, the class~$\mathcal{Z}$ is an atomic class with a single
element of size $1$, the~$\mathcal{N}$ classes are all epsilon
classes containing a single element of size $0$, marking internal nodes, 
and the function $S(z) = \sum_{j \geq 4} s_j z^j$ is the generating
function for simple permutations. 
\label{thm:AA05}
\end{theorem}

Notice that $\mathcal{U}_{\oplus}$ and $\mathcal{U}_{\ominus}$ define combinatorial classes which are in size-preserving bijection. 
In the following, in order to deal with one class instead of two, we replace them by the equivalent class 
$\mathcal{U} = \mathcal{Z}_{\Box} +\, \mathcal{N}_{\circ}\cdot\operatorname{Seq}_{\geq 2}{\mathcal{U}} +\,\mathcal{N}_{\bullet}\cdot\, S(\mathcal{P})$. 
Doing so, we change the labels of the linear nodes having a linear parent (replacing them by $\circ$). 
This does not affect the enumeration of the class. 
Indeed, these labels are determined since a linear node and its linear parent have different labels.

It is not hard to view $\mathcal{U}$ as a family of trees 
not unlike those of studied in Section~\ref{sec:General}: 

\begin{corollary}
The following combinatorial equivalences are true:

\greybox{
\begin{equation*}
\mathcal{P}\equiv \operatorname{Seq}_{\geq 1}{\mathcal{U}} \qquad \text{and} \qquad
\mathcal{U}\equiv \mathcal{Z} + \operatorname{Seq}_{\geq
  2}{\mathcal{U}} + S(\operatorname{Seq}_{\geq 1}{\mathcal{U}}).
\end{equation*}
}
Consequently, $\mathcal{U}$ is in bijection with a class of $\Lambda$-trees, 
or in other words its generating function $U(z)$ satisfies $U(z) = z + \Lambda(U(z))$, 
for $\Lambda(x)= \frac{x^2}{1-x} +\sum_{j\geq 4} s_j \left(\frac{x}{1-x}\right)^j$, 
where $s_j$ is the number of simple permutations of size $j$. 
\end{corollary}
\begin{proof}
This equivalence is derived from Equation~\eqref{eq:initialsystem}, 
the fact that $\mathcal{U} \equiv \mathcal{U}_{\oplus} \equiv \mathcal{U}_{\ominus}$,
and the intermediary equivalence $\mathcal{P} \equiv \mathcal{U}+
\operatorname{Seq}_{\geq 2}{\mathcal{U}}$.
\end{proof}

There is however an important difference between $\mathcal{U}$ and the set of classes that are covered by Theorem~\ref{thm:main}: 
the function $\Lambda$ defined by $\Lambda(x)= \frac{x^2}{1-x} +\sum_{j\geq 4} s_j \left(\frac{x}{1-x}\right)^j$ is not analytic at the origin. 
This is due to $S(z)$, the generating function for simple permutation, not being analytic at the origin. 
This somewhat undesirable property follows from an enumerative
study of simple permutations done by Albert \emph{et al.}~\cite{AlAtKl03}. 
We can, however, make use of their asymptotic enumeration formulas, which we recall below. 

Recall that~$s_n$ denotes the number of simple permutations of size~$n$. 
The sequence $(s_n)$ has label~A111111 in the On-Line Encyclopedia of Integer Sequences~\cite{OEIS-simple}. 
This sequence is not P-recursive, but it does satisfy a simple functional inversion formula (see~\cite{AlAtKl03}), 
and we have calculated exact values of $s_n$ for $n<800$. 
Albert \emph{et al.\/}~\cite{AlAtKl03} determined the following bounds:
\begin{equation}
\label{eq:bounds}
\frac{n!}{e^2}\left(1-\frac{4}{n}\right) \leq s_n\leq \frac{n!}{e^2}\left(1-\frac{4}{n}+\frac{2}{n(n-1)}\right).
\end{equation}
Here are the first few terms in the generating function for simple permutations:
\begin{equation*}
S(z)=2z^4+6z^5+46z^6+338z^7+ 2926z^8+ 28146z^9+ 298526z^{10}+ 3454434z^{11}+\dots
\end{equation*}

Because $S(z)$, and hence $\Lambda(x)$, are not analytic at the origin, 
neither $\mathcal{P}$ nor $\mathcal{U}$~are simple varieties of trees whose analysis is covered by Section~\ref{sec:General}. 
Of course, $\mathcal{P}$ being in bijection with permutations, 
this gives access to a very good understanding of $\mathcal{P}$ 
(in particular, enumeration results and random generation tools). 

However, very shortly we propose a different strategy for studying $\mathcal{P}$: 
we describe a filtration $(\mathcal{P}^{(k)})_{k \geq 4}$
such that each $\mathcal{P}^{(k)}$ is built in a straightforward manner from a simple variety of trees $\mathcal{U}^{(k)}$. 
This makes these subclasses $\mathcal{P}^{(k)}$ easy to analyze,
particularly given the generic analysis we have completed in Section~\ref{sec:General}. 
This is done in Section~\ref{sec:restrictedCIT}. 
Next, we view the entire class $\mathcal{P}$ as the (combinatorial) limit of the filtration. 
As explained in further details in Section~\ref{sec:filtration}, one of our goals is to understand how much this strategy 
can yield concerning the asymptotic behavior of $\mathcal{P}$, from that of the subclasses $\mathcal{P}^{(k)}$.

\section{Prime-Degree Restricted Strong Interval Trees}
\label{sec:restrictedCIT}

The filtration of the class of trees $\mathcal{P}$ that we consider 
consists in bounding the maximal arity of prime nodes. 
As we shall see, the results of Section~\ref{sec:General} are applicable to the subclasses of $\mathcal{P}$ where the arity of prime nodes is bounded.
Our motivation for studying this restriction of strong interval trees is twofold. 

First, as indicated above, this gives an example of a non-analytic class which is the limit of a sequence of families of trees which are almost simple varieties of trees: 
we believe this example is instructive and opens the way to adapting this strategy to study other non-analytic classes, 
as we discuss in Section~\ref{sec:filtration}. 

Our second motivation comes from the study of genome rearrangements, specifically in the model of perfect sorting by reversals. 
Indeed, as shown in~\cite{BeBeChPa07,BoChMiRo11}, the algorithmic complexity of finding an evolutionary scenario in this model 
depends heavily on the maximal arity of prime nodes in the strong interval trees of the permutations that encodes the genomes (recording the order of the genes): 
the smaller this maximal arity, the more efficient the algorithm. 
Based on biological data for mammalian genomes~\cite{ChMcMi11}, it appears that prime nodes occur relatively rarely, and are of small arity. 
In~\cite{BoChMiRo11}, the combinatorics of strong interval trees without any prime nodes was investigated, 
resulting in a better understanding of the so-called \emph{commuting scenarios}. 
Now allowing some prime nodes to occur, but with a bounded arity, we are going a step further in this analysis, 
while focusing on subclasses of strong interval trees that seem to represent well the biological data.

\subsection{The filtration for permutations}

We define the class $\mathcal{P}^{(k)}$ as follows, where $S^{\leq k}(z)=\sum_{j=4}^ks_jz^j$:
\begin{equation*}
\mathcal{P}^{(k)} = \mathcal{Z} + 2\operatorname{Seq}_{\geq 2}{\mathcal{U}^{(k)}} + S^{\leq k}(\mathcal{P}^{(k)})\qquad
\text{and} \qquad
\mathcal{U}^{(k)} = \mathcal{Z} + \operatorname{Seq}_{\geq 2}{\mathcal{U}^{(k)}} + S^{\leq k}(\mathcal{P}^{(k)}).
\end{equation*}

That is, only prime nodes of arity at most $k$ are allowed. 
We refer to the classes denoted by $\mathcal{P}^{(k)}$, as \emph{prime-degree restricted\/} strong interval trees. 

The containment $\mathcal{P}^{(k)}\subset\mathcal{P}^{(k+1)}$ is straightforward, and
since $\mathcal{P}^{(k)}_n = \mathcal{P}_n$ when $k\geq n$, we can
derive the limit of combinatorial classes $\lim_{k\rightarrow \infty}
\mathcal{P}^{(k)}=\mathcal{P}$.

Furthermore, by the same manipulations as for the full class, we derive:
\begin{equation}
\label{Pk-modified}
\mathcal{P}^{(k)} \equiv  \operatorname{Seq}_{\geq 1}{\mathcal{U}^{(k)}}\qquad \text{and} \qquad
\mathcal{U}^{(k)} \equiv \mathcal{Z} + \operatorname{Seq}_{\geq
  2}{\mathcal{U}^{(k)}} + S^{\leq k}(\operatorname{Seq}_{\geq
  1}{\mathcal{U}^{(k)}}).
\end{equation}

Again similarly to the case of the full class, remark that $\mathcal{U}^{(k)}$ is isomorphic to a $\Lambda_k$-tree with
$\Lambda_k(x)=\frac{x^2}{1-x}+ \sum_{j=4}^k s_j \left(\frac{x}{1-x}\right)^j$. 
This class is certainly algebraic and is a simple variety of trees.
The enumerative analysis of Section~\ref{sec:General} applies directly
to these families of trees $\mathcal{U}^{(k)}$, then giving access to enumeration 
and parameter average behavior for $\mathcal{P}^{(k)}$ also, 
even if $\mathcal{P}^{(k)}$ is not itself a simple variety of trees. 
This is done in the remaining part of this section, focusing on applications to the study of genome rearrangements. 
Also, keeping in mind our next goal of letting $k$ go to infinity to recover the class $\mathcal{P}$, 
we would like to preserve $k$ as much as possible in the formulas.

%
%

\subsection{Asymptotic enumeration}

The equations~\eqref{Pk-modified} allow us to directly apply
Theorem~\ref{thm:main} to determine asymptotic formulas for the
coefficients of the generating functions $P^{(k)}$ of the classes $\mathcal{P}^{(k)}$.

\begin{theorem} For fixed $k$, the number of prime-degree restricted
  strong interval trees of size $n$, denoted $P^{(k)}_n$, grows
  asymptotically like
  
\greybox{
\begin{equation}\label{eq:pkn-sim}
P^{(k)}_n  \sim\frac{\gamma_k}{(1 - \tau_k)^2} \rho_k^{-n}n^{-3/2}\quad\text{ as }n\rightarrow \infty, \quad \text{where} \quad \gamma_k=\sqrt{\frac{\rho_k}{2\pi\Lambda_k''(\tau_k)}}.
\end{equation}

Here, $\Lambda_k(x)=\frac{x^2}{1-x}+\sum_{j=4}^k s_j
(\frac{x}{1-x})^j$, $\tau_k$ satisfies
  $1-\Lambda_k'(\tau_k)=0$ and $\rho_k=\tau_k-\Lambda_k(\tau_k)$. 
}
\label{thm:asym_pnk}
\end{theorem}
\begin{proof}
First, we note that since $\sum_{j=4}^k s_j
(\frac{x}{1-x})^j$ is a polynomial in $\frac{x}{1-x}$,
$\Lambda_k$ is certainly analytic at $0$. 
The radius of convergence of $\Lambda_k$ is easily seen to be $1$, and 
$\lim_{x \rightarrow 1^-} \Lambda_k'(x) = +\infty$. 
Hence, Theorem~\ref{thm:main} gives that at $\rho_k$, it holds that:
\[
U^{(k)}(z) = \tau_k - \beta_k \sqrt{1-\frac{z}{\rho_k}} + \mathcal{O}(1-z/\rho_k) \quad \textrm{ with } \beta_k = \sqrt{\frac{2 \rho_k}{\Lambda''(\tau_k)}}.
\]

Note that the enumerative formulas of the first section also yield the 
asymptotic estimate $U^{(k)}_n\sim \gamma_k \rho_k^{-n}n^{-3/2}$ where $\gamma_k=\frac{\beta_k}{2\sqrt{\pi}} = \sqrt{\frac{\rho_k}{2\pi\Lambda''(\tau_k)}}$.

Next, we note that by the first relation in Equation~\eqref{Pk-modified},
$P^{(k)}(z)=\frac{U^{(k)}(z)}{1-U^{(k)}(z)}$. 
By Theorem~\ref{thm:main}, the value of $U^{(k)}(z)$ at its dominant singularity $\rho_k$ is $\tau_k$. 
Moreover, $\tau_k$ is less than the radius of convergence of $\Lambda_k$, \emph{i.e.}, $\tau_k <1$.
So the composition $P^{(k)}(z)=\frac{U^{(k)}(z)}{1-U^{(k)}(z)}$ is subcritical (see \cite[paragraph VI.9]{FlSe09}): 
this implies that the dominant singularity of $P^{(k)}(z)$ is also $\rho_k$, and that at $\rho_k$, we have:
\begin{align*}
P^{(k)}(z) & = \frac{\tau_k - \beta_k \sqrt{1-\frac{z}{\rho_k}} + \mathcal{O}(1-z/\rho_k)}{1 - \tau_k + \beta_k \sqrt{1-\frac{z}{\rho_k}} + \mathcal{O}(1-z/\rho_k)}  \\
& = \frac{\tau_k}{1 - \tau_k} \left( 1 - \frac{\beta_k}{\tau_k} \sqrt{1-\frac{z}{\rho_k}}\right) \left( 1 - \frac{\beta_k}{1- \tau_k} \sqrt{1-\frac{z}{\rho_k}}\right)+ \mathcal{O}(1-z/\rho_k)\\
& = \frac{\tau_k}{1 - \tau_k} - \frac{\beta_k}{(1 - \tau_k)^2} \sqrt{1-\frac{z}{\rho_k}} + \mathcal{O}(1-z/\rho_k).
\end{align*}
The classic Transfer Theorem of asymptotic enumeration then gives 
\[
P^{(k)}_n  \sim \frac{\beta_k}{(1 - \tau_k)^2} \frac{1}{2\sqrt{\pi n^3}} \rho_k^{-n} 
= \frac{\gamma_k}{(1-\tau_k)^2}\rho_k^{-n} n^{-3/2} \textrm{ \quad as claimed.} \qedhere
\]
\end{proof}

Note that along the proof of Theorem~\ref{thm:asym_pnk}, we have seen that $\tau_k <1$, 
an inequality that will be useful in Section~\ref{sec:filtration} 
to bound the asymptotic estimate of Equation~\eqref{eq:pkn-sim}.

\medskip

Table~\ref{tab:rho-tau} contains numeric approximations for $\tau_k$
and $\rho_k$ in the range $k=4\dots 13$.  Using these estimates gives
good asymptotic approximations and the enumerative formulas given in
Equation~\eqref{eq:pkn-sim} converge quickly for fixed $k$.  For
example, when $k=8$, our asymptotic formula is within 2\% of the
correct value at $n=10$.\footnote{Maple code to compute Table~\ref{tab:rho-tau}
is available at https://github.com/marnijulie/strong-interval-trees-maple}
\begin{table}[t]
\centering\small
\begin{tabular}{lll|lll}
$k$ & $\tau_k$ & $\rho_k$ &$k$ & $\tau_k$ &$\rho_k$\\ \hline
$4$&$0.2258458016$ & $0.1454726242$ & $9$& $0.1463252500$& $0.1102193554$ \\
$5$&$0.2043553556$ & $0.1364583031$ &$10$&$0.1375961304$ & $0.1057725121$ \\
$6$&$0.1841224072$ & $0.1277948168$ &$11$&$0.1300393555$&$0.1017629085$\\
$7$&$0.1689470150$ & $0.1210046262$ &$12$&$0.1234001218$& $0.09810173382$\\
$8$&$0.1565912704$ & $0.1152312243$ &$13$&$0.1174959122$& $0.09472586497$\\
\end{tabular}\newline
\caption{Computed approximate values for $\rho_k$ and $\tau_k$ for small values of $k$. }
\label{tab:rho-tau}
\end{table}

\subsection{Parameter analysis}

The average shape of general strong interval trees was described in~\cite{BoChMiRo11}. 
This study is essentially based on Equation \eqref{eq:bounds}, 
which shows that simple permutations make up about $1/9$ of all permutations. 
As a consequence, general strong interval trees have a very flat shape with probability tending to $1$, 
and this shape governs the average case behavior of any tree parameter. 
However, the prime-degree restricted trees are much more rich in this regards, 
and parameter analysis follows from Section~\ref{sec:General}. 

We focus here on some parameters which are to some extent linked to the perfect sorting scenarios for $\sigma$, 
that is, to parsimonious evolutionary scenarios in the model of perfect sorting by reversals 
(see~\cite{BoChMiRo11} for a detailed explanation of this connection). 
We will be specifically interested in 
the number of internal nodes (which is related to the number of reversals in a scenario), 
the number of prime nodes (since the complexity of computing a parsimonious scenario depends on it) 
and the average subtree size (which has a tight connection to the average reversal size). 
These parameters give important insight into the average case analysis of perfect sorting by reversals.

\medskip 

We have seen above that the generating function of $\mathcal{U}^{(k)}$ satisfies 
$U^{(k)}(z) = z + \Lambda_k(U^{(k)}(z))$ with 
$\Lambda_k(x)=\frac{x^2}{1-x}+ \sum_{j=4}^k s_j \left(\frac{x}{1-x}\right)^j$. 
Consequently, $\mathcal{U}^{(k)}$ is a simple variety of trees, 
and this allows to apply directly the results of Section~\ref{sec:General} 
for the average number of internal nodes or the average subtree size sum 
in $\mathcal{U}^{(k)}$ trees. 
The average number of prime nodes in $\mathcal{U}^{(k)}$ trees can also be derived using the general framework developed in Section~\ref{sec:General}. 
Then, the behaviour of these parameters in $\mathcal{P}^{(k)}$ trees is deduced from the already observed identity 
\begin{equation} 
\mathcal{P}^{(k)} = \mathcal{U}^{(k)} + \operatorname{Seq}_{\geq 2}{\mathcal{U}^{(k)}}. \label{eq:fromUtoP}
\end{equation}
Note that even though $\mathcal{P}^{(k)}$ is not a simple variety of trees, 
the behaviour of the studied parameters are of the same order as in such families of trees. 

The results proved in this section are summarized in Table~\ref{tab:summary_Pk}.

\begin{table}[t]
\greybox{
\begin{tabular}{lcll}
  The average number of internal nodes &  \qquad &$\frac{\tau_{k}-\rho_{k}}{\rho_{k}}\,n$\\[4mm]
  The average number of prime nodes  & \qquad &
  $\frac{S^{\leq k}(\tau_{k})}{\rho_{k}}\, n$\\[4mm]
  The average subtree size sum  & \qquad &
  $\frac{\beta_{k}^{2}}{4\rho_{k}\gamma_{k}}\, n^{3/2}$\\[2mm]
\end{tabular}}
\caption{A summary of asymptotic behavior of parameters for trees in $\mathcal{P}^{(k)}$.} 
\label{tab:summary_Pk}
\end{table}

\subsubsection{Number of internal nodes}

Let $U^{(k)}(z,y)$ (resp. $P^{(k)}(z,y)$) be the bivariate generating function of $\mathcal{U}^{(k)}$ trees (resp. $\mathcal{P}^{(k)}$ trees), 
where $z$ counts the size (\emph{i.e.}, the number of leaves) and $y$ counts the number of internal nodes. 
It follows from Eq.~\eqref{eq:fromUtoP} that 
\[
 P^{(k)}(z,y) = U^{(k)}(z,y) + y \cdot \frac{U^{(k)}(z,y)^2}{1-U^{(k)}(z,y)}.
\]
Consequently, we have 
\begin{align}\label{eq:Pderivate_internalNodes}
\frac{\partial}{\partial y} P^{(k)}(z,y) \Big|_{y=1} = \quad & \frac{\partial}{\partial y} U^{(k)}(z,y) \Big|_{y=1} 
+ \frac{2U^{(k)}(z,1)}{1-U^{(k)}(z,1)} \frac{\partial}{\partial y} U^{(k)}(z,y) \Big|_{y=1} \nonumber \\
& + \frac{U^{(k)}(z,1)^2}{1-U^{(k)}(z,1)}
+ \frac{U^{(k)}(z,1)^2}{(1-U^{(k)}(z,1))^2}\frac{\partial}{\partial y} U^{(k)}(z,y) \Big|_{y=1}. 
\end{align}
Near $\rho_k$, we know that 
$
U^{(k)}(z,1) = U^{(k)}(z) = \tau_k - \beta_k \sqrt{1-z/\rho_k} + \mathcal{O}(1-z/\rho_k) 
$.
The weaker estimate $U^{(k)}(z) = \tau_k + \mathcal{O}(\sqrt{1-z/\rho_k})$ gives 
$\frac{1}{1-U^{(k)}(z)} = \frac{1}{1-\tau_k} + \mathcal{O}(\sqrt{1-z/\rho_k})$, 
and these are enough to estimate all rational fractions in $U^{(k)}(z)$ that appear in Eq.\eqref{eq:Pderivate_internalNodes}. 
Moreover, the generating function $\frac{\partial}{\partial y} U^{(k)}(z,y) \Big|_{y=1} $ counts $\mathcal{U}^{(k)}$ trees 
weighted by their number of internal nodes. 
As seen in Section~\ref{sec:General}, it then follows from Lemma~\ref{lm:inv} that, near $\rho_k$, 
\[
\frac{\partial}{\partial y} U^{(k)}(z,y) \Big|_{y=1} = \frac{\beta_k (\tau_k - \rho_k)}{2\rho_k (1-z/\rho_k)^{1/2}} + o\left(\frac{1}{(1-z/\rho_k)^{1/2}}\right).
\]
Combining these asymptotic estimates gives, near $\rho_k$, 
\[
\frac{\partial}{\partial y} P^{(k)}(z,y) \Big|_{y=1} = \frac{\beta_k (\tau_k - \rho_k)}{2\rho_k (1-\tau_k)^{2}}\frac{1}{(1-z/\rho_k)^{1/2}} + o\left(\frac{1}{(1-z/\rho_k)^{1/2}}\right).
\]
Recalling the identity $\gamma_k=\frac{\beta_k}{2\sqrt{\pi}}$  
and the asymptotic behaviour of $[z^n]P^{(k)}(z,1) = P^{(k)}_n$ given in Theorem~\ref{thm:asym_pnk}, 
we deduce that the average number of internal nodes in $\mathcal{P}^{(k)}$ trees is 
\[
\frac{[z^n]\frac{\partial}{\partial y} P^{(k)}(z,y) \Big|_{y=1}}{[z^n]P^{(k)}(z,1)} \sim_{n \to \infty} 
\frac{\beta_k(\tau_k - \rho_k)}{2\rho_k (1-\tau_k)^{2} \sqrt{\pi n}} \rho_k^{-n} \cdot \frac{(1-\tau_k)^2}{\gamma_k\rho_k^{-n}} n^{3/2} = \frac{(\tau_k - \rho_k)}{\rho_k} \cdot n.
\]

\subsubsection{Number of prime nodes}

Like before, let us denote by $U^{(k)}(z,y)$ (resp. $P^{(k)}(z,y)$) the bivariate generating function of $\mathcal{U}^{(k)}$ trees (resp. $\mathcal{P}^{(k)}$ trees), 
counted by size (for $z$) and number of prime nodes (for $y$). 
We know an asymptotic estimates of $U^{(k)}(z,1) = U^{(k)}(z)$ near $\rho_k$, 
and we now apply the method of Section~\ref{sec:General} to compute one for $\frac{\partial}{\partial y} U^{(k)}(z,y) \Big|_{y=1}$. 

For any $\mathcal{U}^{(k)}$ tree $t$, let $\sigma(t) = \xi(t)$ denote the number of prime nodes in $t$ 
and let $\eta(t)$ be $1$ if the root of $t$ is a prime node, $0$ otherwise. 
With the notations of Lemma~\ref{lm:inv}, we have 
\begin{align*}
H(z) & = \sum_{t \in \mathcal{U}^{(k)}} \eta(t) z^{|t|} = S^{\leq k}(U^{(k)}(z)) \text{ \quad where } S^{\leq k}(u) = \sum_{j=4}^{k}s_{j}u^{j}  \\
\text{and \quad} \frac{\partial}{\partial y} U^{(k)}(z,y) \Big|_{y=1} & = \Xi(z) = H(z)\cdot \frac{\partial}{\partial z}U^{(k)}(z) 
= S^{\leq k}(U^{(k)}(z)) \cdot \frac{\partial}{\partial z}U^{(k)}(z).
\end{align*}
The asymptotic estimate of $U^{(k)}(z)$ near $\rho_k$ is 
$U^{(k)}(z) = \tau_k - \beta_k \sqrt{1-z/\rho_k} + \mathcal{O}(1-z/\rho_k)$, 
from which we deduce $U^{(k)}(z)^j = \tau_k^j + o(1)$. 
Moreover, Singular Differentiation gives, near $\rho_k$, 
\[
\frac{\partial}{\partial z}U^{(k)}(z) = \frac{\beta_k}{2\rho_k(1-z/\rho_k)^{1/2}} + o\left(\frac{1}{(1-z/\rho_k)^{1/2}}\right).
\]
Consequently, we obtain that near $\rho_k$, 
\[
\frac{\partial}{\partial y} U^{(k)}(z,y) \Big|_{y=1} = 
S^{\leq k}(\tau_k) \cdot \frac{\beta_k}{2\rho_k} \cdot \frac{1}{(1-z/\rho_k)^{1/2}}  + o\left(\frac{1}{(1-z/\rho_k)^{1/2}}\right).
\]

Now turning to $\mathcal{P}^{(k)}$ trees, Eq.~\eqref{eq:fromUtoP} implies that 
\[
 P^{(k)}(z,y) = U^{(k)}(z,y) + \frac{U^{(k)}(z,y)^2}{1-U^{(k)}(z,y)}.
\]
Differentiation gives 
\[
\frac{\partial}{\partial y} P^{(k)}(z,y) \Big|_{y=1} = \left( 1 + \frac{2U^{(k)}(z,1)}{1-U^{(k)}(z,1)} + \frac{U^{(k)}(z,1)^2}{(1-U^{(k)}(z,1))^2} \right) 
\frac{\partial}{\partial y} U^{(k)}(z,y) \Big|_{y=1}. 
\]
The asymptotic estimates obtained above give that, near $\rho_k$, 
\[
\frac{\partial}{\partial y} P^{(k)}(z,y) \Big|_{y=1} = 
S^{\leq k}(\tau_k) \cdot \frac{\beta_k}{2\rho_k (1-\tau_k)^2} \cdot \frac{1}{(1-z/\rho_k)^{1/2}}  + o\left(\frac{1}{(1-z/\rho_k)^{1/2}}\right).
\]
Finally, we deduce that the average number of prime nodes in $\mathcal{P}^{(k)}$ trees is 
\[
\frac{[z^n]\frac{\partial}{\partial y} P^{(k)}(z,y) \Big|_{y=1}}{[z^n]P^{(k)}(z,1)} \sim_{n \to \infty} 
\frac{S^{\leq k}(\tau_k) \cdot \beta_k }{2\rho_k (1-\tau_k)^2 \sqrt{\pi n}} \rho_k^{-n} \cdot \frac{(1-\tau_k)^2}{\gamma_k\rho_k^{-n}} n^{3/2} 
= \frac{S^{\leq k}(\tau_k)}{\rho_k} \cdot n.
\]
\subsubsection{Subtree size sum}

Again, we denote by $U^{(k)}(z,y)$ (resp. $P^{(k)}(z,y)$) the bivariate generating function of $\mathcal{U}^{(k)}$ trees (resp. $\mathcal{P}^{(k)}$ trees), 
counted by size (for $z$) and number of prime nodes (for $y$). In this case, Eq.~\eqref{eq:fromUtoP} gives 
\[
 P^{(k)}(z,y) = U^{(k)}(z,y) + \frac{U^{(k)}(zy,y)^2}{1-U^{(k)}(zy,y)}.
\] 
As before, we have $U^{(k)}(z,1) = U^{(k)}(z)$. 
Note also that $ \frac{\partial}{\partial z} U^{(k)}(z,y) \Big|_{y=1} = \frac{\partial}{\partial z}U^{(k)}(z)$. 
It follows that 
\begin{align*}
\frac{\partial}{\partial y} P^{(k)}(z,y) \Big|_{y=1} = &
\frac{\partial}{\partial y} U^{(k)}(z,y) \Big|_{y=1}  \\ 
+ &
\left(\frac{2U^{(k)}(z,1)}{1-U^{(k)}(z,1)} + \frac{U^{(k)}(z,1)^2}{(1-U^{(k)}(z,1))^2} \right) 
\left( z  \frac{\partial}{\partial z}U^{(k)}(z) + \frac{\partial}{\partial y} U^{(k)}(z,y) \Big|_{y=1}  \right)
\end{align*}
and we proceed like in the previous cases. 
Near $\rho_k$, the asymptotic estimate of $\frac{\partial}{\partial z}U^{(k)}(z)$ is 
\[
\frac{\partial}{\partial z}U^{(k)}(z) = \frac{\beta_k}{2\rho_k(1-z/\rho_k)^{1/2}} + o\left(\frac{1}{(1-z/\rho_k)^{1/2}}\right), 
\]
and we have seen in Section~\ref{sec:General} that 
\[
\frac{\partial}{\partial y} U^{(k)}(z,y) \Big|_{y=1} = \frac{\beta_k^2}{4\rho_k(1-z/\rho_k)} + o\left(\frac{1}{1-z/\rho_k}\right),
\]
since this function counts $\mathcal{U}^{(k)}$ trees weighted by their substree size sum. 
Consequently, the asymptotic estimate of $\frac{\partial}{\partial y} P^{(k)}(z,y) \Big|_{y=1}$ near $\rho_k$ is
\[
 \frac{\partial}{\partial y} P^{(k)}(z,y) \Big|_{y=1} = \frac{\beta_k^2}{4\rho_k (1-\tau_k)^2(1-z/\rho_k)} + o\left(\frac{1}{1-z/\rho_k}\right).
\]
We conclude that the average value of the subtree size sum in $\mathcal{P}^{(k)}$ trees is 
\[
\frac{[z^n]\frac{\partial}{\partial y} P^{(k)}(z,y) \Big|_{y=1}}{[z^n]P^{(k)}(z,1)} \sim_{n \to \infty} 
\frac{\beta_k^2}{4\rho_k (1-\tau_k)^2} \rho_k^{-n} \cdot  \frac{(1-\tau_k)^2}{\gamma_k\rho_k^{-n}} n^{3/2} 
= \frac{\beta_k^2}{4\rho_k\gamma_k} \cdot n^{3/2}.
\]

\subsection{Random generation}\label{sec:GRT}
Equation~\eqref{Pk-modified} gives immediate access to random sampling
of trees in~$\mathcal{P}^{(k)}$.  Thinking of the
classes~$\mathcal{P}^{(k)}$ as possible models for the biological data
collected in~\cite{ChMcMi11}, it is interesting to generate random
trees in~$\mathcal{P}^{(k)}$, to compare them with the trees obtained
from the data.  In this context, our interest is the global shape of
the trees, and not the particulars of the internal nodes.  We have
produced a Boltzmann generator which generates trees
in~$\mathcal{P}^{(k)}$ of size approximately $10000$ for~$k$ up to
$800$ without generating the simple permutation labels (prime and
linear nodes are however distinguished).  Figure~\ref{fig:random-k}
illustrates a randomly generated tree from~$\mathcal{P}^{(7)}$ with
approximately $1000$ leaves. Remark that the structure is dominated by
prime nodes of arity~$7$.
\begin{figure}[ht]
\includegraphics[width=14cm]{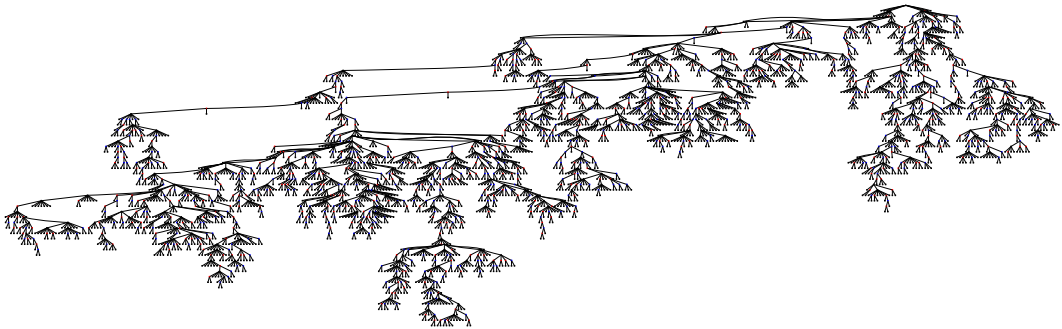}
\caption{A tree from~$\mathcal{P}^{(7)}$ generated uniformly at random}
\label{fig:random-k}
\end{figure}

The results of this random generation experiment are somehow disappointing, 
because the trees generated do not look like the trees obtained from the biological data. 
These have for instance significantly fewer prime nodes and are flatter. 
A careful statistical analysis would however be necessary to properly invalidate our proposed model. 
It would also set a solid basis for the comparison of other proposed models with the data, 
since this statistical analysis could then be reproduced for other subclasses of strong interval trees defined by finer restrictions. 

One of our long term goals on the biological side is to identify the very specific traits which arise
in permutations which encode mammalian genome comparisons, and to provide more adequate models. 
Chauve, McCloskey and Mishna~\cite{ChMcMi11} have taken some preliminary steps in this direction, 
and a reasonable model might use $\mathcal{P}^{(k)}$ trees as components to build realistic trees.
Note that permutations are very often used to model real-world problems, not only in biology: 
for instance also in sorting problems in algorithmic. 
We can therefore view this goal as part of a larger project: 
that of defining subclasses of permutations modeling ``well'' the permutations occurring in these real-world contexts
(using for instance strong interval trees), 
and of proving statistically that these models indeed represent well the data. 

\section{Studying a combinatorial class \emph{via} its filtration} 
\label{sec:filtration}

We now turn to the second aspect of the study of the classes~$\mathcal{P}^{(k)}$: 
understanding the convergence towards the full class $\mathcal{P}$, 
in particular at the level of the asymptotic number of trees in the class. 
The example of $\mathcal{P}$ we have in hand is a particularly instructive one, 
since the limit is known by other means, namely from the correspondence between $\mathcal{P}$ and permutations. 

This example is also meant to illustrate a more general goal, discussed further in Subsection~\ref{subsec:analytic}: 
that of analyzing classes of trees $\mathcal{C}$ via a filtration~$\mathcal{C}^{(k)}$ similar to~$\mathcal{P}^{(k)}$, 
especially when $\mathcal{C}$ fail to be a simple variety of trees 
because the series governing the number of children available (our $\Lambda$) is not analytic.

\subsection{From the asymptotic estimate of $P^{(k)}_n$ to the factorial}
\label{subsec:asymptotic}

Recall that $P^{(k)}_n$ denotes the number of strong interval trees with $n$ leaves and arity of prime nodes at most $k$. 
For any fixed $k$, the asymptotic behavior of~$P^{(k)}_n$ is given in Equation~\eqref{eq:pkn-sim}, 
and is of the form $\gamma \cdot \rho^{-n} \cdot n^{-3/2}$, which is typical for trees. 
However, taking the limit as $k$ tends to infinity, we recover the class of all strong interval trees, 
or equivalently of permutations, and hence an asymptotic behavior in $n! \sim \sqrt{2\pi} e^{-n} n^{n+1/2}$. 
Our goal is to reconcile those estimates. 

\medskip

Most of this section aims at producing an upper bound for the asymptotic estimate of~$P^{(k)}_n$ given in Equation~\eqref{eq:pkn-sim}. 
This is obtained by bounding $\rho_k$ and $\Lambda_k''(\tau_k)$. 
The first ingredient is a more explicit bound for $s_n$, the number of simple permutations of size $n$.

\begin{lemma}\label{lm:sn bound}
For every $n\geq 4$, $s_n \leq \sqrt{2\pi}\,n^{n+1/2}\,e^{-n-2}$.
\end{lemma}
\begin{proof}
This inequality follows from Equation~\eqref{eq:bounds}, stating that 
$s_n \leq \frac{n!}{e^2}\left(1-\frac{4}{n}+\frac{2}{n(n-1)}\right)$,
and the following upper bound on $n!$: $n! \leq \sqrt{2\pi}n^{n+\frac12}e^{-n}\,e^{\frac1{12n}}$.
Combining these two, our claim will follow if we prove that 
$(1-\frac4n+\frac2{n(n-1)})e^{\frac1{12n}} \leq 1$ for $n\geq 4$.
This is equivalent to
$(1-\frac4n+\frac2{n(n-1)})\leq e^{-\frac1{12n}}$. 
And since $1-x\leq e^{-x}$, it is sufficient to prove that 
$1-\frac4n+\frac2{n(n-1)} \leq 1 -\frac1{12n} $, \emph{i.e.}, that 
$4-\frac2{n-1} \geq \frac1{12} $. 
This obviously holds for $n \geq 4$, concluding the proof. 
\end{proof}

From this estimate, the derivations of the bounds on~$\tau_k$ and~$\rho_k$ are relatively straightforward, but technical. 
Working with the value $\ttk=\frac{\tau_k}{1-\tau_k}$ simplifies the expressions. 
To derive those bounds, it is essential to keep in mind this sequence of inequalities, 
which follow from $\tau_k <1$ and $\rho_k = \tau_k - \Lambda_k(\tau_k)$:
\[
0<\rho_k<\tau_k<\ttk<1.
\]

\begin{proposition}[Bounds for $\ttk$]
\label{prop:bounds_tautilde}
For any $\alpha<\frac{e-2}{e-1}$, there exists $k(\alpha)$ such that
for $k>k(\alpha)$
\begin{equation*}
\left(\frac{\alpha}{ks_k}\right)^{\frac{1}{k-1}}<\ttk<\left(\frac{1}{ks_k}\right)^{\frac{1}{k-1}}.
\end{equation*}
Consequently,
\begin{equation*}
\frac{e}{k}\left(\frac{\alpha e^{3}}{\sqrt{2\pi}\,k^{5/2}}\right)^{\frac1{k-1}}<\ttk <\frac{e}{k} \left(\frac{e^3}{\sqrt{2\pi}\,k^{3/2}(k-4)}  \right)^{\frac{1}{k-1}} <\frac{e}{k}.
\end{equation*}
\end{proposition}
Computational evidence suggests that $k(\alpha)=4$, for all $\alpha$ near $\frac{e-2}{e-1}$.

\begin{proof}
The starting point is the equation $\Lambda_k'(x)=1$, satisfied by $\tau_k$. 
Because $\Lambda_k(x) = \frac{x^2}{1-x}+\sum_{j=4}^k s_j (\frac{x}{1-x})^j$, 
it is convenient to consider this equation under the change of variables $y=\frac{x}{1-x}$, \emph{i.e.}, $x=\frac{y}{1+y}$. 
Notice that it implies $\frac{1}{(1-x)^2} = (1+y)^2$. 

Derivation gives $\Lambda_k'(x) = \frac{1}{(1-x)^2} -1 + \frac{1}{(1-x)^2} \sum_{j=4}^k j s_j(\frac{x}{1-x})^{j-1}$, so that 
the equation $\Lambda_k'(x)=1$ can be rewritten as 
\begin{equation}
\label{eq:lambda}
(1+y)^2 -1 + (1+y)^2\sum_{j=4}^k j s_j y^{j-1} = 1\quad \text{ which implies } \quad 
\frac{2-(1+y)^2}{(1+y)^2} = \sum_{j=4}^k j s_j y^{j-1}. 
\end{equation}
The next step towards proving the stated inequalities is 
the fact that for
$0<y<1$,  
$1-5y < \frac{2-(1+y)^2}{(1+y)^2} < 1$,
which is immediately proved by simple manipulations of inequalities. 
Indeed, we now observe that (by definition of $\ttk$), Equation~\eqref{eq:lambda} is satisfied at $y=\ttk$.
Consequently, these inequalities yield an upper and a lower bound for~$\sum_{j=4}^k j s_j \ttk^{j-1}$: 
\begin{equation}
1-5\ttk < \sum_{j=4}^k j s_j \ttk^{j-1} < 1.
\label{eqn:ttk}
\end{equation}

From Equation~\eqref{eqn:ttk}, we get $k s_k \ttk^{k-1} \leq \sum_{j=4}^k j s_j \ttk^{j-1} < 1$, 
from which the upper bound $\ttk < \left(\frac{1}{ks_k}\right)^{\frac{1}{k-1}}$ follows. 
From there, deriving $\ttk <\frac{e}{k} \left(\frac{e^3}{\sqrt{2\pi}\,k^{3/2}(k-4)}  \right)^{\frac{1}{k-1}}$ is then a routine exercise 
using $\frac{k!}{e^2}\left(1-\frac{4}{k}\right) \leq s_k$ (see Equation~\eqref{eq:bounds}) 
and Stirling's inequality $\left(\frac{k}{e}\right)^k \sqrt{2\pi k}\leq k!$. 
This quantity is no larger than $\frac{e}{k}$ as soon as $k\geq 5$. 
This concludes the part of the proof about upper bounds. 

\medskip

For the lower bounds, we start again from Equation~\eqref{eqn:ttk} above. 
We use the inequality
$ 1-5\ttk- \sum_{j=4}^{k-1} j s_j \ttk^{j-1} < k s_k \ttk^{k-1}$, 
and combine it with the bound $0 <\ttk \leq e/k = o(1)$, 
and an upper bound on $\sum_{j=4}^{k-1} j s_j \ttk^{j-1}$ obtained below. 
We split this sum as 
\[
\sum_{j=4}^{k-1} j\,s_j\,\ttk^{j-1} = 
\underbrace{\sum_{j=4}^{k-\iota_k-1} j\,s_j\,\ttk^{j-1}}_{A(k)} +  \underbrace{\sum_{j=k-\iota_k}^{k-1} j\,s_j\,\ttk^{j-1}.}_{B(k)}
\]
where $\iota_k = \lfloor k^{\frac13}\rfloor$. 
Note that $\iota_k$ an non-decreasing integer function of $k$ that tends to infinity and such that $\iota_k=o(\sqrt{k})$. 
Lemmas~\ref{lem:A} and~\ref{lem:B} below prove that 
\[ A(k) = 
\sum_{j=4}^{k-\iota_k-1}
j s_j \ttk^{j-1}=\mathcal{O}\left(\frac{1}{k^3}\right)
\qquad\text{ and that }\qquad
B(k) = \sum_{k-\iota_k}^{k-1}
j s_j \ttk^{j-1}=\frac{1}{e-1} + o(1).
\]
It follows that 
\[
k s_k \ttk^{k-1} > 1 - 5 \ttk - \sum_{j=4}^{k-1} j s_j \ttk^{j-1} = 1 - \frac1{e-1} +o(1) = \frac{e-2}{e-1} + o(1).
\]
Hence for any $\alpha < \frac{e-2}{e-1}$, there exists $k(\alpha)$ such that for any $k\geq k(\alpha)$, 
we have $k\,s_k\,\ttk^{k-1} > \alpha$, 
and therefore 
$\ttk > \left(\frac{\alpha}{k\,s_k}\right)^{\frac1{k-1}}$. 
To conclude the proof, we plug in the upper bound on $s_k$ from Lemma~\ref{lm:sn bound}. 
\end{proof}

\begin{lemma}
The quantity $A(k) = \sum_{j=4}^{k-\iota_k-1} j s_j \ttk^{j-1}$ defined in the proof of Proposition~\ref{prop:bounds_tautilde} 
satisfies $A(k)=\mathcal{O}\left(\frac{1}{k^3}\right)$. 
\label{lem:A}
\end{lemma}
\begin{proof}
It is convenient to define $b_j = j\,s_j$. 
For any $k$, since $\ttk < e/k$, $a_j = b_j \,e^{j-1}\,k^{1-j}$ is an upper bound on $j \,s_j \,\ttk^{j-1}$, so that $A(k) \leq \sum_{j=4}^{k-\iota_k-1} a_j$. 
In what follows, we prove that $\sum_{j=4}^{k-\iota_k-1} a_j =\mathcal{O}\left(\frac{1}{k^3}\right)$ 
(which is enough to conclude, since $A(k)>0$).

We claim that for some integer $j_0$, the sequence $(b_j)_{j\geq j_0}$ is log-convex. 
Indeed, for any $j\geq 6$,
$ \frac{b_j^2}{b_{j-1}b_{j+1}} = \frac{j^2}{(j-1)(j+1)}\frac{s_j^2}{s_{j+1}s_{j-1}}$, and 
Equation~\eqref{eq:bounds} then gives 
\begin{align*}
\frac{b_j^2}{b_{j-1}b_{j+1}}
& \leq \frac{j^2}{(j-1)(j+1)} \frac{j!^2\left(1-\frac{4}{j}+\frac2{j(j-1)}\right)^2}{(j-1)!(j+1)!\left(1-\frac{4}{j+1}\right)\left(1-\frac{4}{j-1}\right)}
= 1-\frac1j + \mathcal{O}\left(\frac1{j^2}\right). 
\end{align*}
In particular, there exists an integer $j_0\geq 6$ 
such that for any $j\geq j_0$, $\frac{b_j^2}{b_{j-1}b_{j+1}}<1$ and therefore
the sequence $(b_j)_{j\geq j_0}$ is log-convex. 

Now, note that for any $k$, $(a_j)_{j\geq j_0}$ is also log-convex, since $\frac{a_j^2}{a_{j-1}a_{j+1}} = \frac{b_j^2}{b_{j-1}b_{j+1}}$ for all $j$. 
The reason for considering the sequence $(b_j)$ instead of $(a_j)$ in the first place is to ensure that $j_0$ does not depend on $k$,
although the definition of $a_j = j\,s_j\,e^{j-1}\,k^{1-j}$ depends on $k$. 

\smallskip

Log-convex sequences are decreasing down to a given minimum then increasing, 
and therefore are bounded from above by the values reached at the extremities. Thus for all 
$ j\in \{j_0,\ldots,k-\iota_k-1\}, a_j \leq \max\{a_{j_0}, a_{k-\iota_k-1}\} \leq a_{j_0} + a_{k-\iota_k-1}$.
Consequently, 
\[
\sum_{j=4}^{k-\iota_k-1} a_j =  \sum_{j=4}^{j_0-1} a_j + \sum_{j=j_0}^{k-\iota_k-1} a_j  \leq \sum_{j=4}^{j_0-1} a_j + k\,a_{j_0} + k\,a_{k-\iota_k-1},
\]
and the result will follow if we find adequate upper bounds on each on these three terms, which we now do. 

\smallskip

For any $j\in\{4,\ldots j_0-1\}$, we have 
$ a_j = j\,s_j\,e^{j-1}\,k^{1-j} \leq j j! \,e^{j-1}\,k^{1-j}\leq j_0 j_0!\,e^{j_0-1}\,k^{-3}
$ so that $\sum_{j=4}^{j_0-1} a_j \leq j_0^2 j_0!\,e^{j_0-1}\,k^{-3} = \mathcal{O}(k^{-3})$.

For the term $k\,a_{j_0}$, we have $k\,a_{j_0} = j_0\,s_{j_0}\,e^{j_0-1}\,k^{2-j_0} =\mathcal{O}(k^{-3})$ since $j_0 \geq 6$. 
Using Lemma~\ref{lm:sn bound} and the fact that for all $x\in (0,1)$, $\log(1-x)<-x$, we obtain the bound for the last term. 
More precisely, we have:
\begin{align*}
k\,a_{k-\iota_k-1} & \leq k^2\cdot s_{k-\iota_k-1}\cdot e^{k-\iota_k-2}\cdot k^{2-k+\iota_k} \\
& \leq k^2\cdot \sqrt{2\pi}\cdot (k-\iota_k-1)^{k-\iota_k-1/2}\cdot e^{-k+\iota_k-1}\cdot e^{k-\iota_k-2}\cdot k^{2-k+\iota_k} \\
& \leq \frac{\sqrt{2\pi}}{e^3}k^{7/2}\cdot\left(1-\frac{\iota_k+1}{k}\right)^{k-\iota_k-1/2} \\
& \leq \frac{\sqrt{2\pi}}{e^3}k^{7/2}\cdot\exp\left((k-\iota_k-1/2)\log\left(1-\frac{\iota_k+1}{k}\right)\right)\\
& \leq \frac{\sqrt{2\pi}}{e^3}k^{7/2}\cdot\exp\left(-\frac{(k-\iota_k-1/2)(\iota_k+1)}k\right).
\end{align*}
The quantity in the exponential is asymptotically equivalent to $-\iota_k = -\lfloor k^{\frac13}\rfloor$. Hence
$k\,a_{k-\iota_k-1}$ decreases super-polynomially fast toward $0$, and is therefore a $\mathcal{O}(k^{-3})$ too.
\end{proof}

\begin{lemma}
The quantity $B(k) = \sum_{k-\iota_k}^{k-1} j s_j \ttk^{j-1}$ defined in the proof of Proposition~\ref{prop:bounds_tautilde} 
satisfies $B(k)=\frac{1}{e-1} +o(1)$. 
\label{lem:B}
\end{lemma}

\begin{proof}
With the change of variable $i=k-j$, we can write $B(k) = \sum_{i=1}^{\iota_k} (k-i)\,s_{k-i}\,\ttk^{k-i-1}$. 
By Lemma~\ref{lm:sn bound} and the upper bound on $\ttk$ proved in Proposition~\ref{prop:bounds_tautilde}, we have
\begin{align*}
(k-i)\,s_{k-i} & \leq \frac{\sqrt{2\pi}(k-i)^{k-i+3/2}}{e^{k-i+2}},\\
\ttk^{k-i-1} & \leq \frac{e^{k-i-1}}{k^{k-i-1}}\left(
\frac{e^{3}}{\sqrt{2\pi}\,k^{3/2}\left(k-4\right)}\right)\cdot\left(
\frac{e^{3}}{\sqrt{2\pi}\,k^{3/2}\left(k-4\right)}\right)^{\frac{-i}{k-1}}.
\end{align*}
Therefore
$
(k-i)\,s_{k-i}\,\ttk^{k-i-1} \leq \left(1-\frac{i}{k}\right)^{k-i+3/2}\left(
e^{-3}\sqrt{2\pi}\,k^{3/2}\left(k-4\right)\right)^{\frac{i}{k-1}}\cdot\frac{1}{1-\frac4{k}}
$. 
Since $i \leq \iota_k$ and $e^{-3}\sqrt{2\pi}\,k^{3/2}\left(k-4\right) \geq 1$ as soon as $k\geq 5$, we obtain that for $k\geq 5$,
\[
(k-i)\,s_{k-i}\,\ttk^{k-i-1} \leq \left(1-\frac{i}{k}\right)^{k-i+3/2}\left(
e^{-3}\sqrt{2\pi}\,k^{3/2}\left(k-4\right)\right)^{\frac{\iota_k}{k-1}}\cdot\underbrace{\frac{1}{1-\frac4{k}}}_{1+\mathcal{O}(\frac1k)}.
\]
Using again that $\log(1-x)<-x$ for $x\in (0,1)$, we have 
\[
\left(1-i/k\right)^{k-i+3/2} = e^{(k-i+3/2)\log\left(1-\frac{i}{k}\right)} \leq e^{ - \frac{(k-i+3/2)i}{k}} = e^{ - i +\frac{i^2}k -\frac{3i}{2k}}.
\]
Recalling that $i \leq \iota_k = \lfloor k^{\frac13}\rfloor $, this gives $\left(1-i/k\right)^{k-i+3/2} \leq e^{-i} \exp(k^{-1/3})= e^{-i}\left(1+o(1)\right)$.
Proceeding similarly, the middle term satisfies
\[
\left(e^{-3}\sqrt{2\pi}\,k^{3/2}\left(k-4\right)\right)^{\frac{\iota_k}{k-1}} = 1 + o(1).
\]
Therefore, we obtain $(k-i)\,s_{k-i}\,\ttk^{k-i-1} \leq e^{-i}\left(1+o(1)\right)$, where the function hidden in the $o(1)$ notation depends on $k$ but not on $i$. 
Consequently, summing over $i$, we obtain 
\[
B(k) \leq \left(\sum_{i=1}^{\lambda_k}e^{-i}\right) \left(1+o(1)\right) \leq \left(\sum_{i=1}^{\infty}e^{-i}\right) \left(1+o(1)\right) 
= \frac{1+o(1)}{e-1} \textrm{ as claimed.} \qedhere
\]
\end{proof}

\begin{theorem}[Bounds for $\rho_k$]
\label{thm:bounds_rho}
There exists a constant $\beta$ such that 
for any $\alpha < \frac{e-2}{e-1}$, there exist 
$k(\alpha,\beta)$ such that for any $k \geq k(\alpha,\beta)$,
\[\frac{e}{k} \left(\frac{\alpha e^3}{\sqrt{2\pi}\,k^{5/2}} \right)^{\frac{1}{k-1}}
\left(1-\frac{\beta}{k}\right)
< \rho_k <
\frac{e}{k} \left(\frac{e^3}{\sqrt{2\pi}\,k^{3/2}(k-4)}  \right)^{\frac{1}{k-1}}.\]
Consequently, $\rho_k=\frac{e}{k}\left(1-\frac{5}{2}\, \frac{\log k}{k} + \mathcal{O}(\frac1k)  \right)$.
\end{theorem}

\begin{proof}
The upper bound is immediate from the bound $\rho_k<\ttk$ and Proposition~\ref{prop:bounds_tautilde}. 

\noindent For the lower bound, we start from $\rho_k = \tau_k-\Lambda_k(\tau_k)$. 
The definitions of $\ttk$ and $\Lambda_k$ give 
$\rho_k=\ttk\left(1-\frac{2\ttk}{1+\ttk}-\sum_{j=4}^k s_j\ttk^{j-1}\right)$. 
Our main step is to deduce from this equality that $\rho_k \geq \ttk (1-\beta/k)$ for some constant $\beta$. 
The lower bound will then follow from Proposition~\ref{prop:bounds_tautilde}. 

As in the proof of Proposition~\ref{prop:bounds_tautilde}, 
we leverage upper bounds on $\ttk$ to build a lower bound on 
$1-\frac{2\ttk}{1+\ttk}-\sum_{j=4}^k s_j\ttk^{j-1}$. 
In this case, we use $\frac{2\ttk}{1+\ttk}\leq2\ttk\leq
2\frac{e}{k}$, and we will bound the summation by splitting the sum at the
same place:
\[\sum\limits_{j=4}^{k} s_j\,\ttk^{j-1} =
\sum\limits_{j=4}^{k-\iota_k-1} s_j\,\ttk^{j-1} +
\sum\limits_{j=k-\iota_k}^{k-1} s_j\,\ttk^{j-1} +
s_k\,\ttk^{k-1}.\]
Even though it is not the same summation, 
we can re-use the bounds from Lemmas~\ref{lem:A} and~\ref{lem:B}. 
Indeed, 
\begin{align*}
\sum\limits_{j=4}^{k-\iota_k-1} s_j\,\ttk^{j-1} & \leq \sum\limits_{j=4}^{k-\iota_k-1} j \, s_j\,\ttk^{j-1} 
= A(k) = \mathcal{O}\left(\frac{1}{k^3}\right) \\
\textrm{and } \sum\limits_{j=k-\iota_k}^{k-1} s_j\,\ttk^{j-1} & \leq \sum\limits_{j=k-\iota_k}^{k-1} \frac{j}{k - \iota_k} s_j\,\ttk^{j-1} 
= \frac{B(k)}{k - \iota_k} = \mathcal{O}\left(\frac{1}{k}\right).
\end{align*}
\noindent Finally, Proposition~\ref{prop:bounds_tautilde} ensures that $k\,s_k\,\ttk^{k-1} \leq 1$, 
and we obtain $\frac{2\ttk}{1+\ttk}+\sum_{j=4}^k s_j\ttk^{j-1} = \mathcal{O}\left(\frac1k\right)$. 
It follows that for some $\beta$, there exists $k(\beta)$ such that when $k \geq k(\beta)$ we have: 
\[
\frac{2\ttk}{1+\ttk}+\sum_{j=4}^k s_j\ttk^{j-1} \leq \frac{\beta}{k} 
\textrm{ \quad and hence \quad } \rho_k \geq \ttk \left( 1-\frac{\beta}{k} \right),
\]
which together with Proposition~\ref{prop:bounds_tautilde} proves the lower bound.

\medskip

To obtain the claimed asymptotic estimate of $\rho_k$, it is enough to observe that both the upper and the lower bound behave like 
$\frac{e}{k}\left(1-\frac{5}{2}\, \frac{\log k}{k} + \mathcal{O}(\frac1k)  \right)$. 
More precisely, 
\begin{align*}
\left(\frac{\alpha e^3}{\sqrt{2\pi}\,k^{5/2}} \right)^{\frac{1}{k-1}} 
& = \exp\left( \frac{\log(k^{-5/2}) + cst}{k-1} \right) 
= \exp\left( -\frac{5}{2}\frac{\log k}{k-1}+ \frac{cst}{k-1} \right) \\
& = 1 - \frac{5}{2}\frac{\log k}{k} +\mathcal{O}\left( \frac{1}{k}\right) \\
\textrm{so that } \left(\frac{\alpha e^3}{\sqrt{2\pi}\,k^{5/2}} \right)^{\frac{1}{k-1}} \left(1-\frac{\beta}{k}\right) & = 
1 - \frac{5}{2}\frac{\log k}{k} +\mathcal{O}\left( \frac{1}{k}\right) \\
\textrm{and }
\left(\frac{e^3}{\sqrt{2\pi}\,k^{3/2}(k-4)}  \right)^{\frac{1}{k-1}} & = 
\exp\left( \frac{\log(k^{-5/2}) + \log (\frac{1}{1-4/k}) + cst}{k-1} \right) \\
& = 1 - \frac{5}{2}\frac{\log k}{k} +\mathcal{O}\left( \frac{1}{k}\right). \qedhere 
\end{align*}
\end{proof}
 
It was known in~\cite{PP2011} that $\rho_k = \frac{e}{k} (1+ o(1))$,
but we are able to produce a more precise estimate. We require this
precision when we consider the limit as $k\rightarrow \infty.$

\medskip

Looking at the asymptotic estimate of $P^{(k)}_n$ provided by Theorem~\ref{thm:asym_pnk}, 
and aiming at obtaining an upper bound on this estimates, 
the only missing piece is a lower bound on $\Lambda_k''(\tau_k)$. 
The definition of $\Lambda_k$ (see Theorem~\ref{thm:asym_pnk}) gives 
\begin{align*}
 \Lambda_k''(x) &= \frac{2}{(1-x)^3}\left( 1 + \sum_{j=4}^{k} j s_j \left(\frac{x}{1-x} \right)^{j-1}+ 
 \frac{1}{2(1-x)}\sum_{j=4}^{k} j (j-1) s_j \left(\frac{x}{1-x} \right)^{j-2} \right) \\
 & \geq \frac{2}{(1-x)^3} \textrm{ for all } x \in (0,1), 
\end{align*}
and therefore the series expansion of $(1-x)^{-3}$ ensures that
$\Lambda_k''(\tau_k) \geq 2+6\ttk$. We could expand this expression
further, and use lower bounds on $\ttk$, but it turns out that for our
purposes, the bound $\Lambda_k''(\tau_k) \geq 2$ is sufficient. 

\medskip

Finally, we have all of the elements to find an upper bound for the asymptotic
estimate of $P^{(k)}_n$.

\begin{theorem}
For any fixed $k$, as $n$ tends to infinity, 
$P^{(k)}_n$ behaves like $\frac{\gamma_k}{(1 - \tau_k)^2} \rho_k^{-n}n^{-3/2}$, 
where  where $\gamma_k = \sqrt{\frac{\rho_k}{2\pi\Lambda_k''(\tau_k)}}$. 
And when $k$ grows to infinity, this estimates is no larger than 

\greybox{
\begin{equation}\label{eq:estimate_upper_bound}
\frac{1}{(1-\frac{e}{k})^2}\sqrt{\frac{e}{4k\pi}}\left(\frac{k}{e}\right)^{n} 
\left(1+\frac{5}{2}\frac{\log k}{k} + \mathcal{O}\left(\frac{1}{k} \right)\right)^n \, n^{-3/2}
\end{equation}
}
\end{theorem}

\begin{proof}
Theorem~\ref{thm:main} ensures that $P^{(k)}_n \sim \frac{\gamma_k}{(1 - \tau_k)^2} \rho_k^{-n}n^{-3/2}$ as $n \rightarrow \infty$. 
From Proposition~\ref{prop:bounds_tautilde} and Theorem~\ref{thm:bounds_rho}, assuming now that $k$ is large enough, we have: 
\begin{itemize}
 \item $\tau_k \leq \ttk \leq \frac{e}{k}$, hence $\frac{1}{(1 - \tau_k)^2} \leq \frac{1}{(1-\frac{e}{k})^2}$; 
 \item $\rho_k \leq \frac{e}{k}$ and $\Lambda_k''(\tau_k) \geq 2$, hence $\gamma_k \leq \sqrt{\frac{e}{4k\pi}}$;
 \item and $\rho_k^{-n} = \left(\frac{e}{k}\right)^{-n} \left(1-\frac{5}{2}\frac{\log k}{k} + \mathcal{O}\left(\frac{1}{k} \right)\right)^{-n}
 = \left(\frac{k}{e}\right)^{n} \left(1+\frac{5}{2}\frac{\log k}{k} + \mathcal{O}\left(\frac{1}{k} \right)\right)^n$. \qedhere
\end{itemize}
\end{proof}

\paragraph{In the limit, Stirling's approximation.}
Our analysis of~$\mathcal{P}$ has brought together two classic asymptotic facts. 
The asymptotic growth of each $\mathcal{P}^{(k)}$ 
is of the form $P^{(k)}_n \sim \gamma \rho^{-n} n^{-3/2}$ for some real valued $\rho$ and $\gamma$. 
(Note that although $\mathcal{P}^{(k)}$ is not a simple variety of trees, 
the asymptotic behaviour of $P^{(k)}_n$ is of the same form as for such families.) 
But for the full class $\mathcal{P}$, the classical Stirling's approximation of $n!$ gives~$P_n \sim \left(\frac{n}{e}\right)^n\sqrt{2\pi n}$. 
Subtle analysis is required to reconcile these two estimates, 
and our upper bound on the asymptotic estimate of $P^{(k)}_n$ allows us to take a first step in this direction.

For any $n$, the strong interval tree of a permutation of size $n$ contains no prime node of arity larger than $n$. 
Thus, if $k\geq n$, $\mathcal{P}^{(k)}_n$ contains all trees corresponding to permutations of size $n$, 
and hence $P^{(k)}_n = n!$ for $k\geq n$. 
Now, forget for a moment that the estimates for $\mathcal{P}^{(k)}_n$ as $n \rightarrow \infty$ 
is valid only for fixed $k$, and consider the expression in~\eqref{eq:estimate_upper_bound} with $k=n$. 
It simplifies as follows:
\[
\frac{1}{(1-\frac{e}{n})^2}\sqrt{\frac{e}{4n\pi}}\left(\frac{n}{e}\right)^{n} 
\left(1+\frac{5}{2}\frac{\log n}{n} + \mathcal{O}\left(\frac{1}{n} \right)\right)^n \, n^{-3/2} 
= \sqrt{\frac{e}{4\pi}} \left(\frac{n}{e} \right)^n \sqrt{n} \cdot (1+o(1)).
\]
This is a constant times Stirling's formula (the constant being $\sqrt{\frac{e}{8 \pi^2}}$). 
And this is encouraging: indeed, even though setting $k=n$ was not justified, 
these purely formal manipulations do reconcile the two asymptotics, up to a constant factor. 

However, if we follow the same route for $P^{(2n)}_n$, which is also $n!$, 
the quantity in~\eqref{eq:estimate_upper_bound} gains an unwanted factor of $2^n$. 
This does not contradict the correctness of our asymptotic form for fixed $k$. 
It rather emphasizes that it is an open problem 
to develop asymptotic formulas when $k$ is a function of $n$, and they go to infinity together. 
This will require 
a very delicate treatment of the bounds, 
a much stronger understanding of how to take the limit as $k\rightarrow\infty$, 
and a return to the analytic inversion and transfer theorems to study how the error terms depend on~$k$. 

\subsection{A simpler case: when $\Lambda$ is analytic}
\label{subsec:analytic}

The example developed above is meant to illustrate a strategy to enumerate classes $\mathcal{C}$ of trees 
whose generating functions satisfy $C(z) = z + \Lambda(C(z))$, in particular in the case where $\Lambda$ is not analytic. 

The method we proposed is to consider a sequence of analytic $\Lambda_k$ 
(obtained for instance by truncations at order $k$) 
such that as formal power series,
$\lim_{k\rightarrow\infty}\Lambda_k = \Lambda$, and to study first the sets $\mathcal{C}^{(k)}$ of $\Lambda_k$-trees.

The next step, which is for the moment not accessible to us,
is to obtain results about $\mathcal{C}$ from what is known on the classes $\mathcal{C}^{(k)}$. 
We view the following questions as particularly interesting in this regard: 
Can we describe conditions so that the limit of the asymptotics of the subclasses tends to the asymptotics of the whole class? 
To which extent are the parameter formulas valid under the limit? 
We hope that our work will help developing techniques to obtain information on $\mathcal{C}$ by letting $k$ go to infinity. 

\medskip

The difficulty here lies in $\Lambda$ being not analytic. 
Notice however that the same filtration by truncations at order $k$ may also be defined when $\Lambda$ is analytic. 
Next, we show that in this case, we obtain the correct asymptotic formula when taking the limit as $k$ tends to infinity, 
\emph{i.e.} that limits in $n$ and $k$ commute.

\medskip

Consider a series $\Lambda(x) = \sum_{i\geq 2} \lambda_i x^i$ with non-negative coefficients. 
And for all $k \geq 2$, define $\Lambda_k(x) = \sum_{i = 2}^k \lambda_i x^i$. 
We denote respectively by $\mathcal{C}$ and $\mathcal{C}^{(k)}$ the classes of trees whose generating functions satisfy
\[
C(z) = z + \Lambda(C(z)) \textrm{\qquad and \qquad} C^{(k)}(z) = z + \Lambda_k(C^{(k)}(z)).
\]
We make the following assumptions: 
$\Lambda$ is analytic at $0$, and denoting $R$ the radius of convergence of $\Lambda$, 
there is a unique solution $\tau \in (0,R)$ to the equation $\Lambda'(x) =1$. 
Then, it follows from Theorem~\ref{thm:main} that $C(z)$ is analytic at $0$, 
has a unique dominant singularity $\rho = \tau - \Lambda(\tau)$, 
and, assuming further that $C(z)$ is aperiodic, that the coefficients of this series 
behave asymptotically like $[z^n]C(z) \sim \sqrt{\frac{\rho}{2\pi\Lambda''(\tau)}}\cdot \frac{\rho^{-n}}{n^{3/2}}$.

\begin{lemma}
For all $k\geq 2$, there exists a unique $\tau_k \in (0,+\infty)$ such that $\Lambda_k'(\tau_k) =1$. 
Moreover, the sequence $(\tau_k)_{k\geq 2}$ is decreasing and converges to $\tau$ as $k$ goes to infinity. 
\end{lemma}

\begin{proof}
Fix some $k\geq 2$. From the definition of $\Lambda_k(x) = \sum_{i= 2}^k \lambda_i x^i$,
it follows that $\Lambda'_k(x)$ is a polynomial with non-negative coefficients, 
increasing from $0$ to $+\infty$ when $x$ varies from $0$ to $+\infty$. 
Moreover, its derivative $\Lambda''_k$ being nowhere zero on $(0,+\infty)$, $\Lambda'_k$ is strictly increasing. 
Therefore, there is a unique positive solution to $\Lambda_k'(x) =1$, that we denote $\tau_k$. 

The fact that the sequence $(\tau_k)_{k\geq 2}$ is decreasing is immediate from 
\[
1 = \Lambda_k'(\tau_k) = \sum_{i = 2}^k i \lambda_i \tau_k^{i-1} \leq \sum_{i= 2}^{k+1} i \lambda_i \tau_k^{i-1} = \Lambda_{k+1}'(\tau_k)
\]
and the fact that $\Lambda_{k+1}'$ is increasing.

The sequence $(\tau_k)_{k\geq 2}$ being decreasing and non-negative, it admits a limit, that we denote $\ell$. 
We want to prove that $\ell = \tau$, \emph{i.e.}, that $\Lambda'(\ell) =1$. 
First, for all $k$, $\ell \leq  \tau_k$, so that $\Lambda'_k(\ell) \leq 1$. 
Moreover, the sequence $(\Lambda'_k(\ell))$ is increasing (we keep adding non-negative terms), 
and thus converges towards a limit that is no larger than $1$. 
This limit being $\Lambda'(\ell)$, we obtain that $\Lambda'(\ell) \leq 1$. 
Second, since $(\tau_k)_{k\geq 2}$ is decreasing towards $\tau < R$, 
we get that for $k$ large enough, $\Lambda'$ is defined in $\tau_k$. 
For any such large $k$, we have 
\[
1 = \Lambda_k'(\tau_k) = \sum_{i= 2}^k i \lambda_i \tau_k^{i-1} \leq \sum_{i\geq 2} i \lambda_i \tau_k^{i-1} = \Lambda'(\tau_k),
\]
and taking the limit in $k$ gives $\Lambda'(\ell) \geq 1$.
\end{proof}

\begin{lemma}
For all $k\geq 2$, define $\rho_k = \tau_k - \Lambda_k(\tau_k)$.
The sequence $(\rho_k)_{k\geq 2}$ converges to $\rho$ as $k$ goes to infinity.
\end{lemma}

\begin{proof}
It is enough to prove that $(\Lambda_k(\tau_k))$ converges to $\Lambda(\tau)$. 
Like in the previous proof, since $(\tau_k)_{k\geq 2}$ is decreasing towards $\tau < R$, 
we get that for $k$ large enough, $\Lambda$ is defined in $\tau_k$. 
For such large $k$, we have 
\[
\sum_{j=2}^k \lambda_k \tau^k \leq \sum_{j=2}^k \lambda_k \tau_k^k \leq \sum_{j\geq 2} \lambda_k \tau_k^k
\textrm{\quad that is to say } \Lambda_k(\tau) \leq \Lambda_k(\tau_k) \leq \Lambda(\tau_k).
\]
Because $(\Lambda_k(\tau))$ and $(\Lambda(\tau_k))$ share the same limit $\Lambda(\tau)$ as $k$ goes to infinity, 
we obtain that $\lim\limits_{k \rightarrow +\infty} \Lambda_k(\tau_k) = \Lambda(\tau)$. 
\end{proof}

From Theorem~\ref{thm:main}, we obtain
$[z^n] C^{(k)}(z) \sim \sqrt{\frac{\rho_k}{2\pi\Lambda_k''(\tau_k)}}\cdot \frac{\rho_k^{-n}}{n^{3/2}}$, 
and the two lemmas above ensure that taking the limit in $k$ in this estimates give 
$\sqrt{\frac{\rho}{2\pi\Lambda''(\tau)}}\cdot \frac{\rho^{-n}}{n^{3/2}}$. 
In addition, $\lim\limits_{k \rightarrow +\infty} C^{(k)}(z) = C(z)$ from which we get 
$[z^n] \lim\limits_{k \rightarrow +\infty} C^{(k)}(z) \sim \sqrt{\frac{\rho}{2\pi\Lambda''(\tau)}}\cdot \frac{\rho^{-n}}{n^{3/2}}$. 
In other words, taking the limit in $k$ in the estimate of the number of trees of size $n$ in $\mathcal{C}^{(k)}$ 
gives the estimates of the number of trees of size $n$ in $\mathcal{C}$.

\section*{Acknowledgments}
This work was partially supported by ANR project \textsc{Magnum} (2010-BLAN-0204), 
NSERC Discovery grant 31-611453, and funding by Universit\'e Paris-Est.

\medskip

We are indebted to Cedric Chauve for his guidance and access to the
mammalian genome data set, prepared by Bradley Jones and Rosemary
McCloskey. Furthermore, Ms. McCloskey wrote the code for the Boltzmann
generator, amongst other extremely useful things.  We thank Carine
Pivoteau for demonstrating her interest in our project at several
stages.  MM is particularly grateful to both LIGM and LaBRI for hosting her
during the course of this work.
\bibliographystyle{plain}

\end{document}